\documentclass[11pts]{amsart}
\usepackage{amssymb}
\usepackage{epsfig}
\usepackage{amscd}
\usepackage{graphicx}
\newtheorem{theorem}{Theorem}[section]
\newtheorem{prop}[theorem]{Proposition}
\newtheorem{cor}[theorem]{Corollary}
\newtheorem{lem}[theorem]{Lemma}
\newtheorem{rem}[theorem]{Remark}
\newtheorem{defn}[theorem]{Definition}
\newtheorem{example}[theorem]{Example}

\def\ra{\rightarrow}
\def\o{\otimes}
\def\b{\beta}
\def\s{\sigma}
\def\a{\alpha}
\def\mc{\mathcal}
\def\lan{\langle}
\def\ran{\rangle}

\def\l{\lambda}
\def\e{\epsilon}
\def\ve{\varepsilon}
\begin{document}
\title{Representation theory for the Kri\v{z} model }
\thanks{2010 AMS Classification Primary: 55R80, 20C30; Secondary: 55P62, 13A50.\\
\indent This research is partially supported by Higher Education Commission, Pakistan.}
\author{Samia Ashraf$^{1}$, Haniya Azam$^{2}$, Barbu Berceanu$^{3}$}
\address{$^{1}$Abdus Salam School of Mathematical Sciences,
 GC University, Lahore-Pakistan.}
\email {samia.ashraf@yahoo.com}
\address{$^{2}$Abdus Salam School of Mathematical Sciences,
 GC University, Lahore-Pakistan.}
\email {centipedes.united@gmail.com}
\address{$^{3}$Abdus Salam School of Mathematical Sciences,
 GC University, Lahore-Pakistan, and
 Institute of Mathematics Simion Stoilow, Bucharest-Romania (Permanent address).}
\email {Barbu.Berceanu@imar.ro}

\maketitle

\maketitle
\pagestyle{myheadings} \markboth{\centerline {\scriptsize
SAMIA ASHRAF,\,\,\,HANIYA AZAM ,\,\,BARBU BERCEANU}}{\centerline {\scriptsize
Representation theory for the Kri\v{z} model}}
\maketitle
\begin{abstract} The natural action of the symmetric group on the configuration spaces $F(X,n)$ induces an action on the Kri\v{z}
model $E(X,n)$. The representation theory of this DGA is studied and a big acyclic subcomplex which is $\mc{S}_n$-invariant is described.
\end{abstract}

\textwidth=13cm
\pagestyle{myheadings}

\section{Introduction}

The ordered configuration space of $n$ points $F(X,n)$ of a topological space $X$ is defined as
$$F(X,n)=\{(x_1,x_2,\ldots,x_n)\in X^n \mid x_i\neq x_j \mbox{ for } i\neq j\}.$$
For $X$ a smooth complex projective variety, I. Kri\v{z} \cite{K} constructed a rational model $E(X,n)$ for $F(X,n)$, a simplified version of the Fulton-MacPherson model \cite{FM}. Recently Lambrechts and Stanley \cite{LaSt} constructed a (quasi)-model for the configuration space of a topological space with Poincar\'{e}
duality cohomology; if such a space is formal, the model of Lambrechts-Stanley is reduced to the Kri\v{z} model and this is the case
of K\"{a}hler manifolds, see \cite{DGMS}. Therefore all the results of this paper could be applied to (simply connected) formal closed
manifolds (with few changes for the odd-dimensional manifolds).

Let us remind the construction of Kri\v{z}. We denote by $p_i^*:H^*(X)\ra H^*(X^n)$ and $p_{ij}^*:H^*(X^2)\ra H^*(X^n)$ (for $i\neq j$) the pullbacks of the obvious projections and by $m$ the complex dimension of $X$ (for cohomology groups we use rational or complex coefficients). The model $E(X,n)$ is defined as follows: as an algebra $E(X,n)$ is isomorphic to the exterior algebra with generators $G_{ij},\,1\leq i,j\leq n$ (of degree $2m-1$) and coefficients in $H^*(X)^{\o n}$ modulo the relations\\
$$\begin{array}{rcll}
  G_{ji}         & = & G_{ij}                    &  \\
  p_j^*(x)G_{ij} & = & p_i^*(x)G_{ij},            & (i< j),\, x\in H^*(X) \\
  G_{ik}G_{jk}   & = & G_{ij}G_{jk}-G_{ij}G_{ik}, & (i<j<k).
\end{array}$$
The differential $d$ is given by $d|_{H^*(X)^{\o n}}=0$ and $d(G_{ij})=p_{ij}^*(\Delta)$, where $\Delta= w\o1+\ldots+1\o w\in H^*(X)\o H^*(X)$ denotes the class of the diagonal and $w\in H^{2m}(X)$ is the fundamental class.

This model is a differential bigraded algebra $E(X,n)=\mathop{\bigoplus}\limits_{k,q}E_q^k(X,n)$: the lower degree $q$ (called the exterior degree) is given by the number of exterior generators $G_{ij}$, and the upper degree $k$ is given by the total degree:
the multiplication is homogeneous $$E_q^k\o E_{q'}^{k'}\ra E_{q+q'}^{k+k'}$$ and the differential has bidegree ${+1}\choose{-1}$
$$d:E_q^k\ra E_{q-1}^{k+1}.$$
In the next definition $G_{I_*J_*}$ is a product of exterior generators $$G_{I_*J_*}=G_{i_1j_1}G_{i_2j_2}\ldots G_{i_qj_q}.$$
\begin{defn}{(Fulton-MacPherson \cite{FM}, Kri\v{z} \cite{K})}
The symmetric group $\mc{S}_n$ acts on $E(X,n)$ by permuting the factors in $H^*(X^n)=H^{*\o n}$ and changing the indices of the exterior generators: for an arbitrary permutation  $\s\in \mc{S}_n$ $$\s( p_1^*(x_{h_1})\dots  p_n^*(x_{h_n})G_{I_*J_*})= p_{\s(1)}^*(x_{h_1})\dots  p_{\s(n)}^*(x_{h_n})G_{\s(I_*)\s(J_*)}.$$
\end{defn}
The action of $\mc{S}_n$ is well defined because the set of relations is invariant under this action.

In the next diagram the nonzero bigraded components $E_q^k(X,n)$ lie in the trapezoid with
vertices $(0,0)$, $(2mn,0)$, $((n-1)(2m-1),n-1)$ and $(n(2m-1)+1, n-1).$ \\
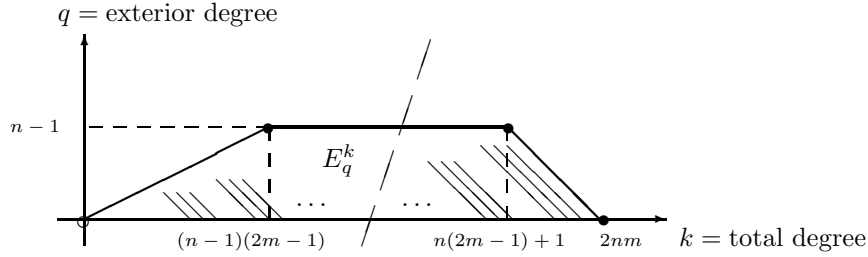
\begin{picture}(80,100)
\put(40,0){ \put(0,10){\vector(1,0){230}} \put(10,0){\vector(0,1){80}}   \put(9.5,9){\circle{4}}
\thicklines  \put(10,10){\line(2,1){70}}  \put(80.5,45){\line(1,0){89}}  \put(170.1,45){\line(1,-1){35.15}}
\begin{scriptsize}
\put(45,0){$(n-1)(2m-1)$}                  \put(142,0){$n(2m-1)+1$}  \put(205,0){$2nm$}  \put(-18,43){$n-1$}
\end{scriptsize}
\thinlines                                  \multiput(80,45)(0,-10){4}{\line(0,-1){5}}    \multiput(170,45)(0,-10){4}{\line(0,-1){5}}
\multiput(80,45)(-9,0){8}{\line(-1,0){5}}   \multiput(170,38)(-5,0){3}{\line(1,-1){28}}   \multiput(150,32)(-5,0){3}{\line(1,-1){22}}
\multiput(70,25)(-5,0){3}{\line(1,-1){15}}  \multiput(50,20)(-5,0){3}{\line(1,-1){10}}
\multiput(77,42)(91,0){2}{$\bullet$}        \put(204,7){$\bullet$}               \put(100,30){$E_q^k$}
\multiput(115,0)(7,21){4}{\line(1,3){5}}    \multiput(90,15)(40,0){2}{$\ldots$}           \put(235,0){$k=\mbox{total\,\,degree}$}
\put(0,85){$q=\mbox{exterior\,\,degree}$}}

\end{picture}
\bigskip

Along the horizontal lines we describe a duality preserving the symmetric action: the dotted median line of the trapezoid is
the axis of symmetry for this duality, see Proposition \ref{prop2.2} in Section \ref{section2}. In the same section we introduce
the combinatorial ``types" of monomials of $E_*^*(X,n)$: these are parameterized by forests in which every tree contains a cohomology class of $X$.
The type decomposition of the bigraded components gives a direct sum of $\mc{S}_n$ submodules, each of these being generated by a unique
element: see Theorem \ref{thm2.8}. In the next section we describe the $\mc{S}_n$ structure of types: explicit decomposition into
irreducible representations in many particular cases and, using the results of Lehrer-Solomon \cite{LeSo}, we compute the character for
the general type. See the Propositions \ref{prop3.1}, \ref{prop3.2}, \ref{prop3.3}, \ref{prop3.4} and Theorem \ref{thm3.15}.

In Section \ref{section4} we present some properties of the differential which are consequences of its $\mc{S}_n$-equivariance. For all complex
projective manifolds, except the projective line, we show that the differential is injective on the ``left side" of the trapezoid:
\begin{prop}\label{prop4}
The differentials in the Kri\v{z} model of a projective manifold different from $\mathbb{C}P^1$ are injective for any $q$ in the interval $[1,n-1]:$
$$d:E_q^{q(2m-1)}(X,n)\rightarrowtail E_{q-1}^{q(2m-1)+1}(X,n).$$
\end{prop}
The top horizontal line has no contribution to the cohomology, too:
\begin{prop}\label{prop5}
The top differentials in the Kri\v{z} model are injective for any $k$ in the interval $[(n-1)(2m-1), n(2m-1)+1]:$
$$d: E_{n-1}^k(X,n)\rightarrowtail E_{n-2}^{k+1}(X,n).$$
\end{prop}
In Section 5 we show that the ``right side" of the trapezoid is an acyclic complex:
\begin{prop}\label{prop6}
All cohomology groups of the subcomplex $$0\ra E_{n-1}^{n(2m-1)+1}(X,n)\ra E_{n-2}^{n(2m-1)+2}(X,n)\ra \ldots \ra E_0^{2nm}(X,n)\ra 0$$ are zero.
\end{prop}
Other (smaller) copies of this subcomplex are contained in the interior of the trapezoid and their sum gives a large acyclic complex which is also
$\mc{S}_n$-equivariant. This subcomplex $E_*^*(w(X,n))$ and the quotient $E_*^*(X,n)\diagup E_*^*(w(X,n))$ are described in Proposition \ref{prop5.7} and
Proposition \ref{prop5.8}; the location of the subcomplex $E_*^*(w(X,n))$ is given in the diagram by the interior lines with slope $-1$.
In \cite{BMP} a different acyclic subcomplex of $E_*^*(X,n)$ is described: this is an ideal, but is not an $\mc{S}_n$-submodule.
The subcomplex $E_*^*(w(X,n))$ is an $\mc{S}_n$-subalgebra, but not an ideal. The right side of the trapezoid, denoted by $E_*^{Top}(X,n)$ in
Section 5, is an acyclic ideal which is also an $\mc{S}_n$-submodule, but it is quite small.

In the last section the simplest and, from the viewpoint of Proposition \ref{prop4}, the exceptional case of
$\mathbb{C}P^1$ is analyzed; we recover and we complete the results of Cohen-Taylor \cite{CT1}, \cite{CT2} and
Feichtner-Ziegler \cite{FZ}, computing in this case the Poincar\'{e} polynomials in two variables $$P_{F(X,n)}(t,s)=\sum_{k,q\geq0}(dim\,H_q^k)t^ks^q.$$
\begin{theorem}\label{thm6}
In the cohomology algebra of the configuration space $F(\mathbb{C}P^1,n)$ ($n\geq4$) the non zero bigraded components are
$$H_q^q\cong H_{q+1}^{q+3}\,\mbox{ for }\,q=0,1,\ldots ,n-3. $$
Its Poincar\'{e} polynomial is $$P_{F(\mathbb{C}P^1,n)}(t,s)=(1+st^3)(1+2st)(1+3st)\ldots (1+(n-2)st).$$
\end{theorem}
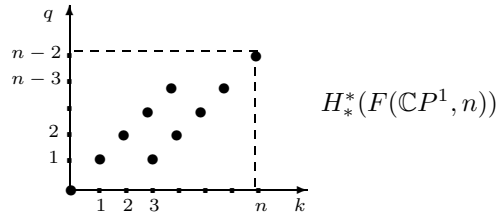
\begin{picture}(360,100)
\put(110,15){\put(0,0){\vector(1,0){90}}   \put(0,0){\vector(0,1){70}}           \multiput(-1.5,10)(0,10){5}{$\centerdot$}
\multiput(10,-1)(10,0){7}{$\centerdot$}    \put(-2,-2.5){$\bullet$}              \multiput(9,9)(20,0){2}{$\bullet$}
\multiput(18,18)(20,0){2}{$\bullet$}       \multiput(27,27)(20,0){2}{$\bullet$}  \multiput(36,36)(20,0){2}{$\bullet$}
\put(68,48){$\bullet$}                     \put(95,30){$H_*^*(F(\mathbb{C}P^1,n))$}
\scriptsize
\put(-8,10){1}    \put(-8,20){2}  \put(-22,40){$n-3$}    \put(-22,50){$n-2$}    \put(-10,65){$q$}
\put(10,-8){1}    \put(20,-8){2}  \put(30,-8){3}         \put(70,-8){$n$}       \put(85,-8){$k$}
\multiput(70,52.5)(0,-6){9}{\line(0,-1){3}}
\multiput(71,52.5)(-6,0){12}{\line(-1,0){3}}  }
\end{picture}

Other extensions and applications of the results of this paper could be found in \cite{AsB} and \cite{AzB}.

For the irreducible ${\mc S}_n$-modules we will use the standard notation (see \cite{FH}): $V(\lambda)$ corresponds to the
partition of $n,\,\, \,\,\lambda\vdash n,$
$\lambda=(\lambda_1\geq \lambda_2\geq\ldots\geq\lambda_t\geq1)$, and also the stable notation (see \cite{CF} or \cite{AAB}):
$V(\mu)_n=V(n-\sum \mu_i,\mu_1,\mu_2,\ldots,\mu_s)$ for $\mu=(\mu_1\geq\mu_2\geq\ldots\geq\mu_s\geq 1)$
satisfying the relation $n-\mathop{\sum}\limits_{i=1}^s \mu_i\geq \mu_1$.

\section{Cohomology classes in the forest}\label{section2}

We fix a (monomial) ordered basis for the cohomology algebra $H^*(X;\mathbb{Q}):$ $\mathcal{B}=\{x_1=1\prec x_2\prec\ldots\prec x_{B}=w\}$,
where $w\in H^{2m}(X;\mathbb{Q})$ is the fundamental class of $X$ and $B=\sum \b_i$ is the sum of Betti numbers;
we choose the order $\prec$ such that the sequence $\{\deg x_i\}_{i=1,B}$ is increasing (not necessarily strictly increasing). Using simple
computations with Diamond lemma (see \cite{Bg}) one can find a monomial basis for the Kri\v{z} model: we denote by
$G_{I_*J_*}=G_{i_1j_1}G_{i_2j_2}\ldots G_{i_qj_q}$ the exterior monomial corresponding to the sequences $I_*=(i_1,\ldots,i_q),\,J_*=(j_1,\ldots,j_q)$,
where $i_a<j_a$ $(a=1,2,\ldots,q)$ and $j_1<j_2<\ldots<j_q$, and by $x_{H_*}=x_{h_1}\o x_{h_2}\o \ldots \o x_{h_n}$ $(x_{h_a}\in \mathcal{B})$
a scalar from $H^{*\o n}$. Then
 $$\{x_{H_*}G_{I_*J_*}\mid x_{h_a}=1\,\,\mbox{if}\,\,a\in J_*,\, \deg x_{H_*}=k-q(2m-1) \}$$
 is a basis of $E_q^k(X,n)$ and we call it the \emph{canonical} (Bezrukavnikov) basis (see \cite{Bz}).

 The next result is obvious:
\begin{prop}
The bigraded components $E_q^k(X,n)$ are invariant under the action of the symmetric group and the differential $d$ is $\mc{S}_n$-equivariant: $$d(\s(x_{H_*}G_{I_*J_*}))=\s(d(x_{H_*}G_{I_*J_*})).$$
\end{prop}
This proposition and the Schur lemma give a splitting of the Kri\v{z} complex into subcomplexes corresponding to the decomposition
of $E_*^*$ into isotypical components $E_*^*(V(\lambda))$, for $\l$ an arbitrary partition of $n$:
$$(E_*^*(X,n),d)=\mathop{\bigoplus}\limits_{\lambda\vdash n}(E_*^*(V(\l)),d_{\l}).$$
The cohomology algebra $H^*(X;\mathbb{Q})$ satisfies Poincar\'{e} duality; denote by $\mathcal{B}^*$ the Poincar\'{e} dual basis
$$\mc{B}^*=\{y^1=w,y^2,\ldots, y^B=1\mid x_iy^j=\delta_{ij}w\}.$$

\begin{prop}\label{prop2.2}
For any $q=0,1,\ldots,n-1$ and any $k$ in the interval of integers $[(2m-1)q,2mn-q]$, there is an isomorphism of $\mc{S}_n$-modules
$$E_q^k(X,n)\cong E_q^{2mn+2q(m-1)-k}(X,n).$$
\end{prop}
\begin{proof}
Define the map $\Phi:E_q^k\ra E_q^{2mn+2q(m-1)-k}$ on the basis by
$$\Phi(x_{H_*}G_{I_*J_*})=x'_{H_*}G_{I_*J_*},$$
where the factors of $x'_{H_*}=x_{h_1}'\o x_{h_2}'\o \ldots \o x_{h_n}'$ are given by
$$x_{h_a}'=\Bigg\{\begin{array}{ll}
                            1(=x_{h_a}) & \mbox{if}\,\, h_a\in J_* \\
                            y^{h_a}   &  \mbox{if} \,\,h_a\,\,\mbox{is not in}\,\, J_*.
                          \end{array}  $$
It is easy to see that $\Phi$ is $\mc{S}_n$-equivariant and the sum of the total degree of $x_{H_*}G_{I_*J_*}$ and the total
degree of $\Phi(x_{H_*}G_{I_*J_*})$ is $2mn+2q(m-1)$.
\end{proof}
Now we associate to monomials in the Kri\v{z} model $E_*^*(X,n)$ \emph{marked graphs}; these are extensions of the graphs
introduced in \cite{LeSo} in order to study the representation theory of the Arnold algebra, the cohomology algebra of $F(\mathbb{C},n)$, see \cite{A}.
To any monomial $G_{I_*J_*}$ from the Arnold algebra $\mc{A}(n)$, Lehrer and Solomon associated a graph $\Gamma$ with
vertices $\{1,2,\ldots,n\}$ and edges $\{i,j\}$ corresponding to the factors $G_{ij}$ of the given monomial.
We associate to the monomial $\mu=x_{h_1}\ldots\o x_{h_n}G_{I_*J_*}$ from $H^*(X)^{\o n}\o \mc{A}(n)$ the graph $\Gamma$ corresponding to $G_{I_*J_*}$, marking the vertices $1,2,\ldots,n$ with the monomials $x_{h_1},\ldots,x_{h_n}$ from the fixed basis $\mc{B}$. Due to the relation $p_j^*(x)G_{ij}  =  p_i^*(x)G_{ij}$, in the image of $\mu$ in the Kri\v{z} algebra $E_*^*(X,n)$ we can move all the factors $x_{h_i}$ corresponding to a given connected component of $\Gamma$ on the smallest index of that component; in this way the marks of the associated graph of the monomial $\mu=x_{h_1}\ldots\o x_{h_n}G_{I_*J_*}$ in $E_*^*(X,n)$ ($x_h=1$ if $h\in J_*$) are the marks of the connected components of the graph $\Gamma$.

\begin{example}
Consider the monomial $\mu=x_{H_*}G_{I_*J_*}\in E_{6}^{*}(X,11)$ given by
$$x_{h_1}\o1\o1\o x_{h_4}\o1\o1\o1\o x_{h_8}\o1\o x_{h_{10}}\o x_{h_{11}}G_{12}G_{23}G_{45}G_{46}G_{47}G_{89}.$$
Its associated marked graph is

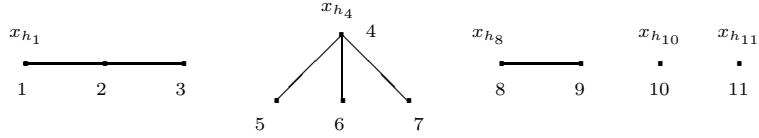
\begin{picture}(60,60)
\multiput(30,29)(30,0){3}{$\centerdot$}       \multiput(210,29)(30,0){4}{$\centerdot$}
\put(149.5,40){$\centerdot$}                    \multiput(125,15)(25,0){3}{$\centerdot$}
\multiput(30.5,30)(30,0){2}{\line(1,0){30}}   \put(211,30){\line(1,0){30}}
\put(151,41){\line(-1,-1){25}}                \put(151,41){\line(0,-1){25}}                   \put(151,41){\line(1,-1){25}}
\scriptsize
\put(28,18){$1$}   \put(58,18){$2$}           \put(88,18){$3$}      \put(160,40){$4$}         \put(118,5){$5$}    \put(148,5){$6$}
\put(178,5){$7$}   \put(209,18){$8$}          \put(239,18){$9$}     \put(267,18){$10$}        \put(297,18){$11$}  \put(25,40){$x_{h_1}$}
\put(143,50){$x_{h_4}$}                       \put(200,40){$x_{h_8}$}                         \put(263,40){$x_{h_{10}}$}
\put(293,40){$x_{h_{11}}$}
\end{picture}

\end{example}
It is obvious how to obtain an element in $E_*^*(X,n)$ starting with a graph with vertices $\{1,2,\ldots,n\}$ and marks
from $H^*(X)$ on the connected components of this graph.

\begin{rem}
\cite{LeSo} If a marked graph contains a cycle, the associated element in $E_*^*(X,n)$ is zero.  Therefore we will consider only
marked forests (all the connected components are trees).
\end{rem}
\begin{proof}Start an induction on the length of the cycle with length 3:
$$G_{ij}G_{ik}G_{jk}=G_{ij}(G_{ij}G_{jk}-G_{ij}G_{ik})=0.$$ To a cycle of length $l+1$ corresponds an element containing as a factor the product $$G_{i_1i_2}G_{i_2i_3}\ldots G_{i_li_{l+1}}G_{i_{l+1}i_1}=G_{i_1i_2}\ldots G_{i_{l-1}i_l}(G_{i_1i_l}G_{i_1i_{l+1}}-G_{i_1i_l} G_{i_li_{l+1}}),$$
both terms having associated graphs with cycles of length $l$.

\end{proof}

If we restrict the correspondence $\{\mbox{monomials in }E_*^*(X,n)\}\ra \{\mbox{marked graphs}\}$ to the canonical (Bezrukavnikov) basis,
we obtain only marked monotonic graphs:
\begin{defn}
A tree with vertices $\{1\leq i_1<i_2<\ldots<i_p\leq n\}$ is \emph{monotonic} if, for any vertex $i_k$, the unique path from $i_1$ to $i_k$ is strictly increasing: \\
\begin{picture}(120,50)
\put(50,0){             \put(70,30){$i_1 <i_a<i_b<\ldots< i_k$}  \put(70,15){$i_1$}  \put(93,15){$i_a$}
\put(118,15){$i_b$}     \put(160,15){$i_k$}                      \multiput(70.5,11.5)(23,0){2}{\line(1,0){23}}
\multiput(70,10.3)(23,0){3}{$\centerdot$}                       \multiput(160,10.3)(23,0){1}{$\centerdot$}
\put(133,10){$\ldots$}} \put(198,10.3){$\centerdot$}            \put(199,11.5){\line(1,0){12}}
\end{picture}
\\(choosing the root $i_1$, the rooted tree is monotonic).
A forest with vertices $\{1,2,\ldots, n \}$ is \emph{monotonic} if all its trees are  monotonic.
\end{defn}
\begin{example}The tree \\
\begin{picture}(120,20)
\multiput(40,2.8)(23,0){3}{$\centerdot$}        \multiput(40.5,4)(23,0){2}{\line(1,0){23}}   \put(40,8){$3$}
\put(63,8){$2$}          \put(85,8){$5$}        \put(100,8){$\mbox{is monotonic but}$}
\end{picture}
\begin{picture}(120,20)
\multiput(60,2.8)(23,0){3}{$\centerdot$}        \multiput(60.5,4)(23,0){2}{\line(1,0){23}}
\put(60,8){$3$}          \put(83,8){$5$}        \put(103,8){$2$}                            \put(118,8){$\mbox{is not.}$}
\end{picture}
\end{example}
\begin{rem}
There is one to one correspondence
$$\{\mbox{monomials in the canonical basis of $E_*^*(X,n)$}\}\leftrightarrow\{\mbox{marked monotonic forests}\}.$$
\end{rem}
\begin{proof} Let us suppose that the graph $\Gamma$ associated to a canonical monomial $G_{I_*J_*}$ ($j_1<j_2<\ldots<j_q,\,i_a<j_a$ for $a=1,2,\ldots, q$) is connected; from its Euler characteristic we find that $card (I_*\cup J_*)=q+1$. If $\Gamma$ is not monotonic, there is a path $i_1-\ldots-j-k-h$ such that $(i_1\leq )j<k>h$, and this corresponds to a forbidden product $G_{jk}G_{hk}$ in $G_{I_*J_*}.$ Conversely, to any monotonic tree (or forest) corresponds a product $G_{I_*J_*}$ from the canonical basis: a vertex $j$, distinct from the minimal vertex $i$ in the same connected component, is joined with a unique vertex $h$, and this is smaller than $j$, namely the second last vertex on the path from $i$ to $j$; therefore $j$ appears only once on the second position, hence in the sequence $J_*$.
\end{proof}

There is an obvious action of the symmetric group $\mc{S}_n$ on the set of marked graphs: the natural action of $\mc{S}_n$ on the set of vertices $\{1,2,\ldots,n\}$ induces an action on the set of edges and an action on the connected components and the corresponding marks.
The set of monotonic marked forests is not $\mc{S}_n$-stable, like the set of monomials in the canonical basis of the Kri\v{z} model; but
$E_*^*(X,n)$ and the $\mathbb{Q}$ vector space generated by marked monotonic forests are $\mc{S}_n$-stable and these two vector spaces will be identified.
\begin{example}In this example and the next one two $\mc{S}_4$-orbits in the $\mathbb{Q}$-span of monotonic marked forests are described: $\oplus$ and $\ominus$ stand for the sum and difference in this vector space and $\tau_i=(i,i+1)$ $(i=1,2,3)$ are the Coxeter generators of $\mc{S}_4$. To save space, the bullet $\bullet$ corresponds to the vertex 1, the root of the tree, and the vertices connected to 1 are, from left to right, written in increasing order; hence

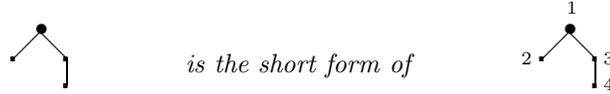
\begin{picture}(360,60)
\put(78,38.5){$\bullet$}                      \put(80,40){\line(-1,-1){10}}     \put(80,40){\line(1,-1){10}}   \put(90,30){\line(0,-1){10}}
\multiput(68.2,28.8)(20,0){2}{$\centerdot$}   \put(88.2,18.8){$\centerdot$}
\put(135,25){$\mbox{is the short form of}$}   \put(278,38.5){$\bullet$}         \put(280,40){\line(-1,-1){10}} \put(280,40){\line(1,-1){10}}
\put(290,30){\line(0,-1){10}}                 \multiput(268.2,28.8)(20,0){2}{$\centerdot$}                     \put(288.2,18.8){$\centerdot$}
\begin{scriptsize} \put(279,47){$1$}  \put(262,28){$2$}  \put(293,28){$3$}  \put(293,18){$4$}  \end{scriptsize}
\end{picture}

\noindent But the next rooted tree is ambiguous
\begin{picture}(70,25)
\put(33,10.5){$\bullet$}                      \put(35,12){\line(-1,-1){10}}     \put(35,12){\line(1,-1){10}}   \put(25,0.5){\line(0,-1){10}}
\multiput(23.2,0.8)(20,0){2}{$\centerdot$}   \put(23.2,-10.2){$\centerdot$}
\end{picture}
therefore we will write

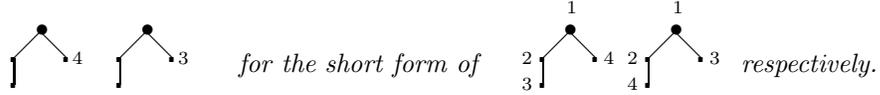
\begin{picture}(360,70)
\put(38,48.5){$\bullet$}                      \put(40,50){\line(-1,-1){10}}     \put(40,50){\line(1,-1){10}}   \put(30,40){\line(0,-1){10}}
\multiput(28.2,38.8)(20,0){2}{$\centerdot$}   \put(28.2,28.8){$\centerdot$}
\put(78,48.5){$\bullet$}                      \put(80,50){\line(-1,-1){10}}     \put(80,50){\line(1,-1){10}}   \put(70,40){\line(0,-1){10}}
\multiput(68.2,38.8)(20,0){2}{$\centerdot$}   \put(68.2,28.8){$\centerdot$}
\put(278,48.5){$\bullet$}                     \put(280,50){\line(-1,-1){10}}    \put(280,50){\line(1,-1){10}}  \put(270,40){\line(0,-1){10}}
\multiput(268.2,38.8)(20,0){2}{$\centerdot$}  \put(268.2,28.8){$\centerdot$}
\put(238,48.5){$\bullet$}                     \put(240,50){\line(-1,-1){10}}    \put(240,50){\line(1,-1){10}}  \put(230,40){\line(0,-1){10}}
\multiput(228.2,38.8)(20,0){2}{$\centerdot$}  \put(228.2,28.8){$\centerdot$}
\put(115,35){$\mbox{for the short form of}$}  \put(305,35){$\mbox{respectively.}$}
\begin{scriptsize}   \put(279,57){$1$}        \put(239,57){$1$}                  \put(262,38){$2$}             \put(222,38){$2$}
\put(293,38){$3$}    \put(253,38){$4$}        \put(262,28){$4$}                  \put(222,28){$3$}
\put(52,38){$4$}     \put(92,38){$3$}         \end{scriptsize}
\end{picture}

\noindent For the same reason the (unique) mark $x_h\in \mc{B}$ is omitted. \\
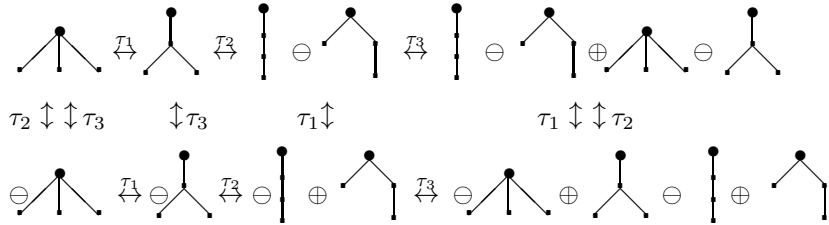
\begin{center}
\begin{picture}(200,100)
\put(0,0){     \put(-42,68){$\bullet$}   \put(-40,70){\line(-1,-1){15}}                \put(-40,70){\line(0,-1){14.5}}
\put(-40,70){\line(1,-1){15}}            \multiput(-56.15,55)(15,0){3}{$\centerdot$}   \put(-20,60){$\stackrel{\tau_1}{\leftrightarrow}$}
\put(0,75){$\bullet$}                    \put(2.2,78){\line(0,-1){12}}                 \put(1,65){$\centerdot$}
\put(2.4,65){\line(-1,-1){10}}           \put(2.4,65){\line(1,-1){10}}                 \multiput(-8.5,54)(19.2,0){2}{$\centerdot$}
\put(18,60) {$\stackrel{\tau_2}{\leftrightarrow}$}                                     \put(35,76.5){$\bullet$}
\put(37.5,78){\line(0,-1){7}}            \put(36.2,68){$\centerdot$}                   \put(37.5,71){\line(0,-1){9}}
\put(36.2,60){$\centerdot$}              \put(37.5,63){\line(0,-1){9}}                 \put(36.2,52){$\centerdot$}
\put(48,60){$\ominus$}                   \put(68,75){$\bullet$}                        \put(70.5,76){\line(-1,-1){10}}
\put(70.5,76){\line(1,-1){10}}           \multiput(59.1,65)(19.2,0){2}{$\centerdot$}   \put(79.7,65){\line(0,-1){10}}
\put(78.3,53){$\centerdot$}              \put(90,60) {$\stackrel{\tau_3}{\leftrightarrow}$}
\put(73,0){   \put(35,76.5){$\bullet$}   \put(37.5,78){\line(0,-1){7}}                 \put(36.2,68){$\centerdot$}
\put(37.5,71){\line(0,-1){9}}            \put(36.2,60){$\centerdot$}                   \put(37.5,63){\line(0,-1){9}}
\put(36.2,52){$\centerdot$} }            \put(121,60){$\ominus$}                       \put(75,0){  \put(68,75){$\bullet$}
\put(70.5,76){\line(-1,-1){10}}          \put(70.5,76){\line(1,-1){10}}                \multiput(59.1,65)(19.2,0){2}{$\centerdot$}
\put(79.7,65){\line(0,-1){10}}           \put(78.3,53){$\centerdot$}   }               \put(160,60){$\oplus$}
\put(222,0){   \put(-42,68){$\bullet$}   \put(-40,70){\line(-1,-1){15}}                \put(-40,70){\line(0,-1){14.5}}
\put(-40,70){\line(1,-1){15}}            \multiput(-56.15,55)(15,0){3}{$\centerdot$} } \put(200,60){$\ominus$}
\put(220,0){   \put(0,75){$\bullet$}     \put(2.2,78){\line(0,-1){12}}                 \put(1,65){$\centerdot$}
\put(2.4,65){\line(-1,-1){10}}           \put(2.4,65){\line(1,-1){10}}                 \multiput(-8.5,54)(19.2,0){2}{$\centerdot$}  }  }
\put(0,6){   \put(0,-60){                \put(-42,68){$\bullet$}                       \put(-40,70){\line(-1,-1){15}}
\put(-40,70){\line(0,-1){14.5}}          \put(-40,70){\line(1,-1){15}}                 \multiput(-56.15,55)(15,0){3}{$\centerdot$} }
\put(5,-60){         \put(-23,60){$\stackrel{\tau_1}{\leftrightarrow}$}                \put(0,75){$\bullet$}
\put(2.2,78){\line(0,-1){12}}            \put(1,65){$\centerdot$}                      \put(2.4,65){\line(-1,-1){10}}
\put(2.4,65){\line(1,-1){10}}            \multiput(-8.5,54)(19.2,0){2}{$\centerdot$} } \put(20,0) {$\stackrel{\tau_2}{\leftrightarrow}$}
\put(7,-60){\put(26,60){$\ominus$}       \put(35,76.5){$\bullet$}                      \put(37.5,78){\line(0,-1){7}}
\put(36.2,68){$\centerdot$}              \put(37.5,71){\line(0,-1){9}}                 \put(36.2,60){$\centerdot$}
\put(37.5,63){\line(0,-1){9}}            \put(36.2,52){$\centerdot$}                   \put(47,60){$\oplus$}
\put(68,75){$\bullet$}                   \put(70.5,76){\line(-1,-1){10}}               \put(70.5,76){\line(1,-1){10}}
\multiput(59.1,65)(19.2,0){2}{$\centerdot$} \put(79.7,65){\line(0,-1){10}}             \put(78.3,53){$\centerdot$} }
\put(94,0) {$\stackrel{\tau_3}{\leftrightarrow}$}   \put(-59,0){$\ominus$}             \put(-6,0){$\ominus$}
\put(109,0){$\ominus$}                   \put(149,0){$\oplus$}                         \put(188,0){$\ominus$}
\put(214,0){$\oplus$}                    \put(170,-60){  \put(35,76.5){$\bullet$}      \put(37.5,78){\line(0,-1){7}}
\put(36.2,68){$\centerdot$}              \put(37.5,71){\line(0,-1){9}}                 \put(36.2,60){$\centerdot$}
\put(37.5,63){\line(0,-1){9}}            \put(36.2,52){$\centerdot$}                   \put(68,75){$\bullet$}
\put(70.5,76){\line(-1,-1){10}}          \put(70.5,76){\line(1,-1){10}}                \multiput(59.1,65)(19.2,0){2}{$\centerdot$}
\put(79.7,65){\line(0,-1){10}}           \put(78.3,53){$\centerdot$}                   \put(-42,68){$\bullet$}
\put(-40,70){\line(-1,-1){15}}           \put(-40,70){\line(0,-1){14.5}}               \put(-40,70){\line(1,-1){15}}
\multiput(-56.15,55)(15,0){3}{$\centerdot$}    \put(0,75){$\bullet$}                   \put(2.2,78){\line(0,-1){12}}
\put(1,65){$\centerdot$}                 \put(2.4,65){\line(-1,-1){10}}                \put(2.4,65){\line(1,-1){10}}
\multiput(-8.5,54)(19.2,0){2}{$\centerdot$}}    \put(100,-25){                         \put(-138,55){$\updownarrow$}
\put(-147,55){$\updownarrow$}            \put(-131,54){$\tau_3$}                       \put(-159,54){$\tau_2$}
\put(-98,55){$\updownarrow$}             \put(-92,54){$\tau_3$}}                       \put(300,-25){ \put(-138,55){$\updownarrow$}
\put(-147,55){$\updownarrow$}            \put(-131,54){$\tau_2$}                       \put(-159,54){$\tau_1$}
\put(-250,54){$\tau_1$}                  \put(-241,55){$\updownarrow$} }}
\end{picture}
\end{center}
In the first example we have a ``small" orbit: the dimension of the representation is $3$ and it is isomorphic to $V(2,1,1)$. In the second
one we will see a ``complete" orbit: the dimension is $6$ and the representation is $V(3,1)\oplus V(2,1,1)$.
\end{example}

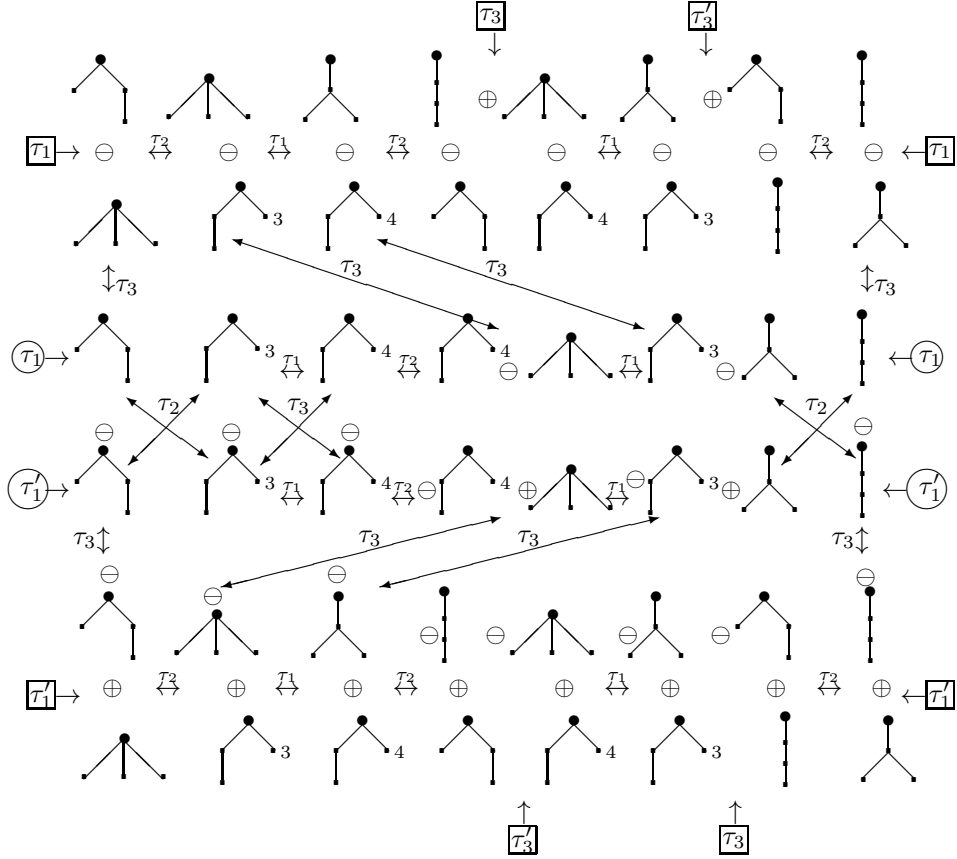
\begin{example}
\mbox{}
\begin{center}
\begin{picture}(300,330)   
\put(40,170){ \put(-5,45){ \put(-108,0){     \put(68,75){$\bullet$}                    \put(70.5,76){\line(-1,-1){10}}
\put(70.5,76){\line(1,-1){10}}           \multiput(59.1,65)(19.2,0){2}{$\centerdot$}   \put(79.7,65){\line(0,-1){10}}
\put(78.3,53){$\centerdot$} }            \put(43,0){  \put(-42,68){$\bullet$}          \put(-40,70){\line(-1,-1){15}}
\put(-40,70){\line(0,-1){14.5}}          \put(-40,70){\line(1,-1){15}}                 \multiput(-56.15,55)(15,0){3}{$\centerdot$} }
\put(47,0){  \put(0,75){$\bullet$}       \put(2.2,78){\line(0,-1){12}}                 \put(1,65){$\centerdot$}
\put(2.4,65){\line(-1,-1){10}}           \put(2.4,65){\line(1,-1){10}}                 \multiput(-8.5,54)(19.2,0){2}{$\centerdot$} }
\put(52,0){ \put(35,76.5){$\bullet$}     \put(37.5,78){\line(0,-1){7}}                 \put(36.2,68){$\centerdot$}
\put(37.5,71){\line(0,-1){9}}            \put(36.2,60){$\centerdot$}                   \put(37.5,63){\line(0,-1){9}}
\put(36.2,52){$\centerdot$}}             \put(170,0){ \put(-42,68){$\bullet$}          \put(-40,70){\line(-1,-1){15}}
\put(-40,70){\line(0,-1){14.5}}          \put(-40,70){\line(1,-1){15}}                 \multiput(-56.15,55)(15,0){3}{$\centerdot$}}
\put(167,0){ \put(0,75){$\bullet$}       \put(2.2,78){\line(0,-1){12}}                 \put(1,65){$\centerdot$}
\put(2.4,65){\line(-1,-1){10}}           \put(2.4,65){\line(1,-1){10}}                 \multiput(-8.5,54)(19.2,0){2}{$\centerdot$}}
\put(140,0){ \put(68,75){$\bullet$}      \put(70.5,76){\line(-1,-1){10}}               \put(70.5,76){\line(1,-1){10}}
\multiput(59.1,65)(19.2,0){2}{$\centerdot$}   \put(79.7,65){\line(0,-1){10}}           \put(78.3,53){$\centerdot$} }
\put(213,0){ \put(35,76.5){$\bullet$}    \put(37.5,78){\line(0,-1){7}}                 \put(36.2,68){$\centerdot$}
\put(37.5,71){\line(0,-1){9}}            \put(36.2,60){$\centerdot$}                   \put(37.5,63){\line(0,-1){9}}
\put(36.2,52){$\centerdot$}}}            
\multiput(-45,85)(47,0){2}{$\ominus$}    \multiput(46,85)(40,0){6}{$\ominus$}     \put(-25,85){$\stackrel{\tau_2}{\leftrightarrow}$}
\put(20,85){$\stackrel{\tau_1}{\leftrightarrow}$}                                 \put(65,85){$\stackrel{\tau_2}{\leftrightarrow}$}
\put(145,85){$\stackrel{\tau_1}{\leftrightarrow}$}                                \put(225,85){$\stackrel{\tau_2}{\leftrightarrow}$}
\put(100,105){$\oplus$}                  \put(185,105){$\oplus$}                  
\put(3,-3){ \put(-42,68){$\bullet$}      \put(-40,70){\line(-1,-1){15}}                \put(-40,70){\line(0,-1){14.5}}
\put(-40,70){\line(1,-1){15}}            \multiput(-56.15,55)(15,0){3}{$\centerdot$}   \put(-63,0){ \put(68,75){$\bullet$}
\put(70.5,76){\line(-1,-1){10}}          \put(70.5,76){\line(1,-1){10}}                \multiput(59.1,65)(19.2,0){2}{$\centerdot$}
\put(60.5,65){\line(0,-1){10}}           \put(59.1,53){$\centerdot$}                   \put(82.5,63){\scriptsize{$3$}} }
\put(-20,0){ \put(68,75){$\bullet$}      \put(70.5,76){\line(-1,-1){10}}               \put(70.5,76){\line(1,-1){10}}
\multiput(59.1,65)(19.2,0){2}{$\centerdot$}  \put(60.5,65){\line(0,-1){10}}            \put(59.1,53){$\centerdot$}
\put(82.5,63){\scriptsize{$4$}} }        \put(20,0){  
\put(68,75){$\bullet$}                   \put(70.5,76){\line(-1,-1){10}}               \put(70.5,76){\line(1,-1){10}}
\multiput(59.1,65)(19.2,0){2}{$\centerdot$}  \put(79.7,65){\line(0,-1){10}}            \put(78.3,53){$\centerdot$} }
\put(60,0){ \put(68,75){$\bullet$}       \put(70.5,76){\line(-1,-1){10}}               \put(70.5,76){\line(1,-1){10}}
\multiput(59.1,65)(19.2,0){2}{$\centerdot$}  \put(60.5,65){\line(0,-1){10}}            \put(59.1,53){$\centerdot$}
\put(82.5,63){\scriptsize{$4$}} }        \put(100,0){ \put(68,75){$\bullet$}           \put(70.5,76){\line(-1,-1){10}}
\put(70.5,76){\line(1,-1){10}}           \multiput(59.1,65)(19.2,0){2}{$\centerdot$}   \put(60.5,65){\line(0,-1){10}}
\put(59.1,53){$\centerdot$}              \put(82.5,63){\scriptsize{$3$}} }             \put(173,0){ \put(35,76.5){$\bullet$}
\put(37.5,78){\line(0,-1){7}}            \put(36.2,68){$\centerdot$}                   \put(37.5,71){\line(0,-1){9}}
\put(36.2,60){$\centerdot$}              \put(37.5,63){\line(0,-1){9}}                 \put(36.2,52){$\centerdot$} }
\put(247,0){ \put(0,75){$\bullet$}       \put(2.2,78){\line(0,-1){12}}                 \put(1,65){$\centerdot$}
\put(2.4,65){\line(-1,-1){10}}           \put(2.4,65){\line(1,-1){10}}                 \multiput(-8.5,54)(19.2,0){2}{$\centerdot$} }
\multiput(-45,40)(287,0){2}{$\updownarrow$}   \multiput(-40,38)(287,0){2}{$\tau_3$}
\multiput(5,57)(55,0){2}{\vector(3,-1){100}}  \multiput(5,57)(55,0){2}{\vector(-3,1){1}}
\multiput(45,45)(55,0){2}{$\tau_3$}           
\multiput(-5,-50)(0,-50){2}{             \put(-110,0){  \put(68,75){$\bullet$}         \put(70.5,76){\line(-1,-1){10}}
\put(70.5,76){\line(1,-1){10}}           \multiput(59.1,65)(19.2,0){2}{$\centerdot$}   \put(79.7,65){\line(0,-1){10}}
\put(78.3,53){$\centerdot$} }            \put(-61,0){ \put(68,75){$\bullet$}           \put(70.5,76){\line(-1,-1){10}}
\put(70.5,76){\line(1,-1){10}}           \multiput(59.1,65)(19.2,0){2}{$\centerdot$}   \put(60.5,65){\line(0,-1){10}}
\put(59.1,53){$\centerdot$}              \put(82.5,63){\scriptsize{$3$}}}              \put(-17,0){\put(68,75){$\bullet$}
\put(70.5,76){\line(-1,-1){10}}          \put(70.5,76){\line(1,-1){10}}                \multiput(59.1,65)(19.2,0){2}{$\centerdot$}
\put(60.5,65){\line(0,-1){10}}           \put(59.1,53){$\centerdot$}                   \put(82.5,63){\scriptsize{$4$}}}
\put(28,0){ \put(68,75){$\bullet$}       \put(70.5,76){\line(-1,-1){10}}               \put(70.5,76){\line(1,-1){10}}
\multiput(59.1,65)(19.2,0){2}{$\centerdot$}   \put(60.5,65){\line(0,-1){10}}           \put(59.1,53){$\centerdot$}
\put(82.5,63){\scriptsize{$4$}}}         \put(177,0){ \put(-42,68){$\bullet$}          \put(-40,70){\line(-1,-1){15}}
\put(-40,70){\line(0,-1){14.5}}          \put(-40,70){\line(1,-1){15}}                 \multiput(-56.15,55)(15,0){3}{$\centerdot$}}
\put(107,0){ \put(68,75){$\bullet$}      \put(70.5,76){\line(-1,-1){10}}               \put(70.5,76){\line(1,-1){10}}
\multiput(59.1,65)(19.2,0){2}{$\centerdot$}   \put(60.5,65){\line(0,-1){10}}           \put(59.1,53){$\centerdot$}
\put(82.5,63){\scriptsize{$3$}} }        \put(210,0){ \put(0,75){$\bullet$}            \put(2.2,78){\line(0,-1){12}}
\put(1,65){$\centerdot$}                 \put(2.4,65){\line(-1,-1){10}}                \put(2.4,65){\line(1,-1){10}}
\multiput(-8.5,54)(19.2,0){2}{$\centerdot$} } \put(210,0){ \put(35,76.5){$\bullet$}    \put(37.5,78){\line(0,-1){7}}
\put(36.2,68){$\centerdot$}              \put(37.5,71){\line(0,-1){9}}                 \put(36.2,60){$\centerdot$}
\put(37.5,63){\line(0,-1){9}}            \put(36.2,52){$\centerdot$} }}                
\multiput(105,5)(82,0){2}{$\ominus$}     \multiput(22,5)(128,0){2}{$\stackrel{\tau_1}{\leftrightarrow}$}
\put(66,5){$\stackrel{\tau_2}{\leftrightarrow}$}                                       
\multiput(-35,-3)(50,0){2}{\vector(4,-3){30}}  \multiput(-35,-3)(50,0){2}{\vector(-4,3){1}}
\multiput(-10,-3)(50,0){2}{\vector(-1,-1){25}} \multiput(-10,-3)(50,0){2}{\vector(1,1){2}}
\put(210,-3){\vector(4,-3){30}}          \put(210,-3){\vector(-4,3){1}}                \put(237,-3){\vector(-1,-1){25}}
\put(237,-3){\vector(1,1){2}}            \put(-24,-7){$\tau_2$}  \put(25,-7){$\tau_3$} \put(221,-7){$\tau_2$}
\multiput(-48,-18)(48,0){2}{$\ominus$}   \put(45,-18){$\ominus$} \put(239,-16){$\ominus$}  
\multiput(22,-43)(123,0){2}{$\stackrel{\tau_1}{\leftrightarrow}$}             \put(64,-43){$\stackrel{\tau_2}{\leftrightarrow}$}
\put(74,-40){$\ominus$}               \put(112,-40){$\oplus$} \put(189,-40){$\oplus$} \put(153,-36){$\ominus$}  
\put(-5,-155){ \put(-108,0){             \put(68,75){$\bullet$}                       \put(70.5,76){\line(-1,-1){10}}
\put(70.5,76){\line(1,-1){10}}           \multiput(59.1,65)(19.2,0){2}{$\centerdot$}  \put(79.7,65){\line(0,-1){10}}
\put(78.3,53){$\centerdot$} }            \put(43,0){  \put(-42,68){$\bullet$}         \put(-40,70){\line(-1,-1){15}}
\put(-40,70){\line(0,-1){14.5}}          \put(-40,70){\line(1,-1){15}}                \multiput(-56.15,55)(15,0){3}{$\centerdot$} }
\put(47,0){ \put(0,75){$\bullet$}        \put(2.2,78){\line(0,-1){12}}                \put(1,65){$\centerdot$}
\put(2.4,65){\line(-1,-1){10}}           \put(2.4,65){\line(1,-1){10}}                \multiput(-8.5,54)(19.2,0){2}{$\centerdot$} }
\put(52,0){ \put(35,76.5){$\bullet$}     \put(37.5,78){\line(0,-1){7}}                \put(36.2,68){$\centerdot$}
\put(37.5,71){\line(0,-1){9}}            \put(36.2,60){$\centerdot$}                  \put(37.5,63){\line(0,-1){9}}
\put(36.2,52){$\centerdot$} }            \put(170,0){ \put(-42,68){$\bullet$}         \put(-40,70){\line(-1,-1){15}}
\put(-40,70){\line(0,-1){14.5}}          \put(-40,70){\line(1,-1){15}}                \multiput(-56.15,55)(15,0){3}{$\centerdot$} }
\put(167,0){ \put(0,75){$\bullet$}       \put(2.2,78){\line(0,-1){12}}                \put(1,65){$\centerdot$}
\put(2.4,65){\line(-1,-1){10}}           \put(2.4,65){\line(1,-1){10}}                \multiput(-8.5,54)(19.2,0){2}{$\centerdot$} }
\put(140,0){ \put(68,75){$\bullet$}      \put(70.5,76){\line(-1,-1){10}}              \put(70.5,76){\line(1,-1){10}}
\multiput(59.1,65)(19.2,0){2}{$\centerdot$}  \put(79.7,65){\line(0,-1){10}}           \put(78.3,53){$\centerdot$} }
\put(213,0){ \put(35,76.5){$\bullet$}    \put(37.5,78){\line(0,-1){7}}                \put(36.2,68){$\centerdot$}
\put(37.5,71){\line(0,-1){9}}            \put(36.2,60){$\centerdot$}                  \put(37.5,63){\line(0,-1){9}}
\put(36.2,52){$\centerdot$} } }          \put(0,-200){ \multiput(-45,85)(47,0){2}{$\oplus$}
\multiput(46,85)(40,0){6}{$\oplus$}      \put(-25,85){$\stackrel{\tau_2}{\leftrightarrow}$}
\put(20,85){$\stackrel{\tau_1}{\leftrightarrow}$}            \put(65,85){$\stackrel{\tau_2}{\leftrightarrow}$}
\put(145,85){$\stackrel{\tau_1}{\leftrightarrow}$}           \put(225,85){$\stackrel{\tau_2}{\leftrightarrow}$}
\multiput(100,105)(-25,0){2}{$\ominus$}  \put(150,105){$\ominus$}                     \put(185,105){$\ominus$}
\multiput(-46,128)(86,0){2}{$\ominus$}   \put(-7,120){$\ominus$}                      \put(240,127){$\ominus$} }
\multiput(-47,-59)(287,0){2}{$\updownarrow$}                 \multiput(-56,-58)(287,0){2}{$\tau_3$}
\multiput(105,-48)(60,0){2}{\vector(-4,-1){105}}             \multiput(105,-48)(60,0){2}{\vector(4,1){1}}
\multiput(52,-57)(60,0){2}{$\tau_3$}      
\put(3,-202){  \put(-42,68){$\bullet$}   \put(-40,70){\line(-1,-1){15}}               \put(-40,70){\line(0,-1){14.5}}
\put(-40,70){\line(1,-1){15}}            \multiput(-56.15,55)(15,0){3}{$\centerdot$}  \put(-63,0){ \put(68,75){$\bullet$}
\put(70.5,76){\line(-1,-1){10}}          \put(70.5,76){\line(1,-1){10}}               \multiput(59.1,65)(19.2,0){2}{$\centerdot$}
\put(60.5,65){\line(0,-1){10}}           \put(59.1,53){$\centerdot$}                  \put(82.5,63){\scriptsize{$3$}} }
\put(-20,0){ \put(68,75){$\bullet$}      \put(70.5,76){\line(-1,-1){10}}              \put(70.5,76){\line(1,-1){10}}
\multiput(59.1,65)(19.2,0){2}{$\centerdot$}   \put(60.5,65){\line(0,-1){10}}          \put(59.1,53){$\centerdot$}
\put(82.5,63){\scriptsize{$4$}} }        \put(20,0){                                  
\put(68,75){$\bullet$}                   \put(70.5,76){\line(-1,-1){10}}              \put(70.5,76){\line(1,-1){10}}
\multiput(59.1,65)(19.2,0){2}{$\centerdot$}  \put(79.7,65){\line(0,-1){10}}           \put(78.3,53){$\centerdot$}}
\put(60,0){ \put(68,75){$\bullet$}       \put(70.5,76){\line(-1,-1){10}}              \put(70.5,76){\line(1,-1){10}}
\multiput(59.1,65)(19.2,0){2}{$\centerdot$}  \put(60.5,65){\line(0,-1){10}}           \put(59.1,53){$\centerdot$}
\put(82.5,63){\scriptsize{$4$}} }        \put(100,0){ \put(68,75){$\bullet$}          \put(70.5,76){\line(-1,-1){10}}
\put(70.5,76){\line(1,-1){10}}           \multiput(59.1,65)(19.2,0){2}{$\centerdot$}  \put(60.5,65){\line(0,-1){10}}
\put(59.1,53){$\centerdot$}              \put(82.5,63){\scriptsize{$3$}}  }           \put(173,0){ \put(35,76.5){$\bullet$}
\put(37.5,78){\line(0,-1){7}}            \put(36.2,68){$\centerdot$}                  \put(37.5,71){\line(0,-1){9}}
\put(36.2,60){$\centerdot$}              \put(37.5,63){\line(0,-1){9}}                \put(36.2,52){$\centerdot$} }
\put(247,0){ \put(0,75){$\bullet$}       \put(2.2,78){\line(0,-1){12}}                \put(1,65){$\centerdot$}
\put(2.4,65){\line(-1,-1){10}}           \put(2.4,65){\line(1,-1){10}}                \multiput(-8.5,54)(19.2,0){2}{$\centerdot$} } } }
\multiput(-70,83)(340,0){2}{\framebox(10,10){$\tau_1$}}      \put(-60,85){$\ra$}      \put(260,85){$\leftarrow$}
\multiput(-70,10)(340,0){2}{\circle{12}} \multiput(-73,8)(340,0){2}{$\tau_1$}         \put(-64,7){$\ra$}
\put(255,7){$\leftarrow$}                \multiput(-70,-40)(340,0){2}{\circle{15}}    \multiput(-73,-42)(340,0){2}{$\tau_1'$}
\put(-64,-43){$\ra$}                     \put(253,-43){$\leftarrow$}                  \multiput(-70,-123)(340,0){2}{\framebox(10,10){$\tau_1'$}}
\put(-60,-121){$\ra$}                    \put(260,-121){$\leftarrow$}                 \put(100,134){\framebox(10,10){$\tau_3$}}
\put(104,126){$\downarrow$}              \put(180,134){\framebox(10,10){$\tau_3'$}}   \put(184,126){$\downarrow$}
\put(112,-178){\framebox(10,10){$\tau_3'$}} \put(115,-165){$\uparrow$}                \put(192,-178){\framebox(10,10){$\tau_3$}}
\put(195,-165){$\uparrow$} }
\end{picture}
\end{center}
\end{example}

\mbox{}

\noindent The second example suggests the next definition

\begin{defn}\label{def 2.6}
We say that two monotonic marked forests are of \emph{the same type}, $(\Gamma,H_*)\thicksim({\Gamma}',H'_*)$, if there is
a permutation $\sigma\in \mc{S}_n$ which induces a bijection between the connected components of the two graphs, preserving
the number of elements of the corresponding components and their marks.
\end{defn}
It is clear that marked monotonic forests in the same $\mc{S}_n$-orbit are of the same type, but not conversely. A complete system
of representatives for this equivalence relation is given by forests of bamboos:

\begin{picture}(360,60)
\put(5,30){$(\Gamma_{L_*},H_*):$}            \multiput(70,30)(15,0){2}{$\centerdot$}      \multiput(160,30)(70,0){1}{$\centerdot$}
\put(185,30){$\centerdot$}                   \multiput(220,30)(30,0){1}{$\centerdot$}     \multiput(130,30)(30,0){1}{$\centerdot$}
\put(235,30){$\ldots$}                       \put(105,25){$\ldots$}                       \put(195,25){$\ldots$}                       \multiput(265,30)(70,0){2}{$\centerdot$}     \multiput(70.5,31.5)(30,0){2}{\line(1,0){30}}\put(160.5,31.5){\line(1,0){60}} 
\put(265.5,31.5){\line(1,0){70}}             \put(70,45){$x_{h_1}$}                       \put(160,45){$x_{h_2}$}
\put(265,45){$x_{h_t}$}                      \put(285,30){$\centerdot$}                   \put(300,25){$\ldots$}
\begin{scriptsize}
\put(70,20){1}     \put(85,20){2}            \put(128,20){$L_1$}  \put(155,20){$L_1+1$}   \put(215,20){$L_2$}
\put(255,20){$L_{t-1}+1$}                    \put(332,20){$L_t$}
\end{scriptsize}
\end{picture}

\noindent where the sequence of lengths of the bamboos $L_*=(\l_1,\l_2,\ldots,\l_t)$ is decreasing $\l_1\geq \l_2\geq\ldots\geq \l_t$,
$L_i=\l_1+\l_2+\dots +\l_i$, $L_t=n$, and, for equal lengths, the marks are in a decreasing order. We split the $\mc{S}_n$-modules $E_q^k(X,n)$
into smaller pieces using the type of the associated monotonic marked forests.
\begin{defn}\label{def 1}
For a given marked forest of bamboos $(\Gamma_{L_*},H_*)$ we define \emph{the subspace of type} $(L_*,H_*)$ as the linear span of monomials with the associated marked graph of type $(\Gamma_{L_*},H_*)$; this will be denoted by $E_*^*(L_*,H_*)$.
\end{defn}
\begin{example}
For the type $(\Gamma_{(3,1,1)},(h_1,h_2,h_3))$, where $x_{h_2}\succ x_{h_3}$,

\begin{picture}(360,60)
\put(10,30){$(\Gamma_{L_*},H_*):$}            \multiput(100,30)(30,0){5}{$\centerdot$}      \put(98,18){$1$}
\put(128,18){$2$}                             \put(158,18){$3$}                             \put(188,18){$4$}
\put(218,18){$5$}                             \multiput(100.5,31.5)(30,0){2}{\line(1,0){30}}  \put(100,40){$x_{h_1}$}
\put(186,40){$x_{h_2}$}                       \put(216,40){$x_{h_3}$}
\end{picture}
the associated space $E_*^*(L_*,H_*)$ is of dimension 40 and its canonical basis is given by the monomials
\[
\begin{array}{cc}
  x_{h_1}\o1\o1\o x_{h_2}\o x_{h_3}G_{12}G_{13}, & x_{h_1}\o1\o1\o x_{h_3}\o x_{h_2}G_{12}G_{13}, \\
  x_{h_1}\o1\o1\o x_{h_2}\o x_{h_3}G_{12}G_{23}, & x_{h_1}\o1\o1\o x_{h_3}\o x_{h_2}G_{12}G_{23}, \\
  \ldots & \ldots \\
  x_{h_2}\o x_{h_3}\o x_{h_1}\o1\o1G_{34}G_{45}, & x_{h_3}\o x_{h_2}\o x_{h_1}\o1\o1G_{34}G_{45}.
\end{array}
\]
If in this example $x_{h_2}=x_{h_3}$, the dimension of $E_*^*(L_*,H_*)$ is 20.
\end{example}
To a given type $(L_*,H_*)$, $L_*=(\l_1,\dots,\l_t)$, $H_*=(h_1,\dots,h_t)$, we will associate two integers:
$$|L_*|=\mathop{\sum}\limits_{i=1}^{t}(\l_i-1)\,\mbox{ and }\,|H_*|=\mathop{\sum}\limits_{h=1}^t\deg (x_h).$$

\begin{theorem}\label{thm2.8}
The bihomogenous components $E_q^k(X,n)$ can be decomposed into a direct sum of monogenic $\mc{S}_n$-submodules
$$E_q^k(X,n)=\mathop{\mathop{\bigoplus}\limits_{|L_*|=q}}\limits_{|H_*|=k-q(2m-1)}E_q^k(L_*,H_*)$$

In particular, the multiplicities of the irreducible $\mc{S}_n$-submodules of each term
$E_q^k(L_*,H_*)=\mathop{\bigoplus}\limits_{\lambda\vdash n}m_{\lambda}V(\lambda)$ satisfy the
relations $m_{\lambda}\leq \dim V(\lambda)$.
\end{theorem}
\begin{proof}
The $\mc{S}_n$-module $E_q^k(X,n)$ is the direct sum $\bigoplus E_q^k(L_*,H_*)$ by the very definition of monomials of type $(L_*,H_*)$. We will show that:

a) for any type $(L_*,H_*)$, the vector space $E(L_*,H_*)$ is $\mc{S}_n$-stable;

b) a generator of the $\mc{S}_n$-module $E_*^*(L_*,H_*)$ is the monomial corresponding to the marked monotonic bamboo $(L_*,H_*)$:
$$\mu_{(L_*,H_*)}=p^*_1(x_{h_1})p^*_{L_1+1}(x_{h_2})\ldots p^*_{L_{t-1}+1}(x_{h_t})\overline{G_{1L_1}}\cdot \overline{G_{L_1+1,L_2}}\cdot \ldots \cdot \overline{G_{L_{t-1}+1,L_t}}$$
(here $\overline{G_{ab}}=G_{a,a+1}G_{a+1,a+2}\ldots G_{b-1,b}$).

The last claim of the theorem is a consequence of the $\mc{S}_n$-equivariant surjection
$$\mathbb{Q}[\mc{S}_n]\ra E(L_*,H_*),\,\,\,\,\, \sigma\mapsto \sigma\cdot \mu_{(L_*,H_*)}.$$

To prove a) it is enough to consider the action of the transpositions $\tau_i=(i,i+1)$ on the tree $T$ containing the vertices $i,i+1$ or on the disjoint union of  two trees, $T'$ and $T''$, containing $i$ and $i+1$ respectively. In the case of one tree, the transform $\tau_i T$ is again a monotonic tree (and $\tau_i T\sim T$) if the path from 1 to $i+1$ does not contain $i$. Otherwise the monotonic path $1-\ldots - h-i-(i+1)$ is transformed into $1-\ldots - h-(i+1)-i$ and the corresponding factor $G_{h,i+1}G_{i,i+1}$ should be replaced by $G_{h,i}G_{i,i+1}-G_{h,i}G_{h,i+1}$. The resulting monomials have monotonic trees of the same type with $T$. The case of two trees, $i$ in $T'$ and $i+1$ in $T''$, is simpler: $\tau_i(T'\sqcup T'')$ is a union of two monotonic trees ($h<i<j$ is equivalent to $h<i+1<j$) and obviously this union is of the same type with $T'\sqcup T''$.

It is enough to prove b) for a monotonic tree: we will show by induction on $n$ that the bamboo $B_n=\mu_{(L_*=(n),H_*=(x_h))}$ generate the module $\mbox{Res}_{\mc{S}_{n-1}}^{\mc{S}_n}E_*^*(X,n)$. If $\sigma$ is a permutation in $\mc{S}_{n-1}$, we denote by $\tilde{\sigma}$ its extension to $\mc{S}_n$: $\tilde{\sigma}(n)=n.$ In the case $n=3$ there is a unique monotonic tree which is not the monotonic bamboo $B_3$, but this belongs to the $\mc{S}_2$-orbit of $B_3$: \\
\begin{picture}(120,20)
\multiput(40,2.8)(23,0){3}{$\centerdot$}      \multiput(40.5,4)(23,0){2}{\line(1,0){23}}    \put(40,8){$2$}
\put(63,8){$1$}          \put(85,8){$3$}      \put(100,8){$=\tau_1($}                       \put(95,0){
\multiput(40,2.8)(23,0){3}{$\centerdot$}      \multiput(40.5,4)(23,0){2}{\line(1,0){23}}    \put(40,8){$1$}
\put(63,8){$2$}          \put(85,8){$3$}      \put(100,8){ $)= \tau_1(B_3)$  }              \put(155,8){and $\tau_1\in \mc{S}_2$.} }
\end{picture} \\
Let us consider a monotonic tree $T_n$ with $n$ vertices and its monotonic subtree $T_{n-1}=T_n\setminus \{n\}$ (by monotonicity, the vertex $n$ is connected to a unique other vertex, $h$). From the set of permutations $\pi\in \mc{S}_{n-1}$ with the property that $\pi(T_{n-1})$ is still monotonic we choose one such that $\pi(h)=j$ is maximal. Now we start a second induction on $n-j$, the number of vertices of $\pi(T_n)$ lying on the branches starting from $j$. If this number is equal to 1, then $j=n-1$ and, by induction on $n$ there are permutations $\sigma_a$ in $\mc{S}_{n-2}$ and constants $c_a\in \mathbb{Q}$ such that $$\pi(T_{n-1})=\mathop{\sum}\limits_{a}c_a\sigma_a(B_{n-1}).$$ The extension of a permutation $\tilde{\sigma}$ of a permutation in $\mc{S}_{n-2}$ does not change the edge $(n-1)-n$ and we find $\tilde{\pi}(T_n)= \mathop{\sum}\limits_{a}c_a\tilde{\sigma}_a(B_{n})$ and therefore $$T_n=\mathop{\sum}\limits_{a}c_a\widetilde{\pi^{-1}\sigma}_a(B_{n}),\,\mbox{where }\,\pi^{-1}\sigma_a\in\mc{S}_{n-1}.$$
If $j$ is less than $n-1$, then $j+1$ is connected to $j$ (by the maximality condition); applying the transposition $\tau_j$, we obtain a non-monotonic tree $\tau_j\pi(T_n)$ (if $j\neq1$). The sequence $k-(j+1)-j$ should be replaced: we obtain a difference of monotonic trees $T_n'-T_n'',$ in each of them $n$ is connected with $j+1$. The second induction will give two expansions $$T_n'=\mathop{\sum}\limits_{p}c_p'\tilde{\sigma}_p'(B_{n}),\, \,\,\,\,T_n''=\mathop{\sum}\limits_{q}c_q''\tilde{\sigma}_q''(B_{n})$$ with $\sigma_p',\,\sigma_q''\in\mc{S}_{n-1}$. Therefore $T_n$ is in $\mathbb{Q}[\mc{S}_{n-1}](B_{n})$: $$T_n=\mathop{\sum}\limits_{p}c_p'\widetilde{\pi^{-1}\tau_j\sigma_p'}(B_{n})-\mathop{\sum}\limits_{q}c_q''\widetilde{\pi^{-1}\tau_j\sigma_q''}(B_{n}).$$
In the case $j=1$ the transposition $\tau_1$ transforms $\pi(T_n)$ into a monotonic tree where $n$ is connected with 2 and we can apply the second induction.
\end{proof}
\begin{cor}
The $\mc{S}_{n-1}$ orbit of the monomial $x_{h}\o1\o\ldots 1G_{12}G_{23}\ldots G_{n-1,n} $ coincides with $E_{n-1}^{(n-1)(2m-1)+|x_h|}(n,h) $ and
$$  Res_{\mc{S}_{n-1}}^{\mc{S}_n}E_{n-1}^{(n-1)(2m-1)+|x_h|}(n,h)\cong \mathbb{Q}[\mc{S}_{n-1}]. $$
\end{cor}
\begin{proof}
In the proof of the theorem we obtained a surjective $\mc{S}_{n-1}$-map
$$ \mathbb{Q}[\mc{S}_{n-1}]\rightarrow Res_{\mc{S}_{n-1}}^{\mc{S}_n}E_{n-1}^{(n-1)(2m-1)+|x_h|}(n,h), \quad \s \mapsto \s\cdot B_n. $$
On the other hand the number of monotonic trees with $n$ vertices is $(n-1)!$.
\end{proof}

\section{The $\mc{S}_n$-module $E_q^k(L_*,H_*)$}\label{section3}

In this section we will study the symmetric structure of the modules $E_q^k(L_*,H_*)$ using two methods. We extend the main result of Lehrer and Solomon \cite{LeSo} and we will give a general formula, but in an implicit form; next, in some particular cases, we will present explicit computations of the irreducible components and their multiplicities, using direct methods.

The symmetric structure of the bottom horizontal line is completely elementary: the types are given
by $L_*=(1^{(n)})$, $H_*=(h_1^{(m_1)}, h_2^{(m_2)}, \ldots, h_t^{(m_t)})$

\begin{picture}(20,50)
\put(55,22){\multiput(0,0)(30,0){2}{$\centerdot$}    \put(40,0){$\ldots$}                           \put(80,0){$\ldots$}
\put(120,0){$\ldots$}                                \multiput(150,0)(30,0){2}{$\centerdot$}
\begin{scriptsize}
\put(0,-10){1}            \put(30,-10){2}            \put(150,-10){$n-1$}                           \put(180,-10){$n$}
\put(0,10){$x_{h_1}$}     \put(30,10){$x_{h_1}$}     \put(150,10){$x_{h_t}$}                        \put(180,10){$x_{h_t}$}
\put(0,-12){$\mathop{\underbrace{\,\,\,\,\,\,\,\,\,\,\,\,\,\,\,\,\,\,\,\,\,\,\,\,\,\,\,\,\,\,\,\,}}\limits_{m_1}$}
\put(135,-12){$\mathop{\underbrace{\,\,\,\,\,\,\,\,\,\,\,\,\,\,\,\,\,\,\,\,\,\,\,\,\,\,\,\,\,\,\,\,\,\,\,\,\,}}\limits_{m_t}$}
\end{scriptsize}}
\end{picture}\\
(here $m_1,m_2,\ldots,m_t$ are multiplicities of the elements $x_{h_1}\succ x_{h_2}\succ\ldots\succ x_{h_t}$ in $\mc{B}$ and $m_1+\ldots+m_t=n$).

\begin{example}
In the case of distinct marks (all the multiplicities $m_i$ are equal to 1) we obtain the largest possible ``type" submodule:
$$E_0^{|H_*|}(1^{(n)},(h_1,h_2,\ldots,h_n))\cong \mathbb{Q}[\mc{S}_n].$$
\end{example}
\begin{proof}
As $\sigma^{-1}(x_{1}\o\ldots\o x_{n})=\pm x_{\sigma(1)}\o\ldots\o x_{\sigma(n)}$, the character of this module is given by $\chi_{E_0^{|H_*|}}(\sigma)=\left\{ \begin{array}{cc}
                                      n! & \sigma=id \\
                                      0 & \sigma\neq id.
                                    \end{array}
\right.$
\end{proof}
\begin{example}
In the case of a unique mark ($m_1=n$)
$$E_0^{n|x_h|}(1^{(n)},h^{(n)})\cong\left\{\begin{array}{lc}
                                       V(n) & \mbox{if $|x_h|$ is even or }n=1 \\
                                       V(1,1,\ldots,1) & \mbox{if $|x_h|$ is odd and }n\geq 2.
                                     \end{array}
\right.$$
\end{example}
\begin{proof}
This is the consequence of the relation $$\tau_i(x_h\o\ldots\o x_h)= (-1)^{|x_h|}(x_h\o\ldots\o x_h).$$
\end{proof}
For the general type corresponding to the discrete graph we will use the notation $$V^{\epsilon(m,h)}=\left\{ \begin{array}{ll}
                                            V(m) & \mbox{if $m=1$ or $|x_h|$ is even} \\
                                            V(1,1,\ldots,1) & \mbox{if $m\geq 2$ and $|x_h|$ is odd.}
                                          \end{array}
                                          \right.$$
\begin{prop}\label{prop3.1}
The $\mc{S}_n$ structure of the type $(L_*,H_*)$, where $L_*=(1^{(n)})$, $H_*=(h_1^{(m_1)}, h_2^{(m_2)},\ldots,h_t^{(m_t)})$ is given by $$E_0^{|H_*|}(1^{(n)},H_*)\cong\mbox{Ind}^{\mc{S}_n}_{\mc{S}_{m_1}\times \mc{S}_{m_2}\times\ldots\times \mc{S}_{m_t}}
V^{\e(m_1,h_1)}\o V^{\e(m_2,h_2)}\o \ldots\o V^{\e(m_t,h_t)}.$$
\end{prop}
\begin{proof}
The symmetric group $\mc{S}_n$ acts transitively on the set of 1-dimensional spaces $\mathbb{Q}\lan x_1\o \ldots\o x_n\ran$, where $m_1$ factors (on different positions) coincide with $x_{h_1}$, ..., $m_t$ factors coincide with $x_{h_t}$. The subgroup leaving invariant the 1-dimensional subspace $\mathop{\underbrace{x_{h_1}\o \ldots\o x_{h_1} }}\limits_{m_1}\o\ldots\o \mathop{\underbrace{x_{h_t}\o \ldots\o x_{h_t} }}\limits_{m_t}$ is the direct product $\mc{S}_{m_1}\times \mc{S}_{m_2}\times\ldots\times \mc{S}_{m_t}$ (with the obvious notation: $\mc{S}_{m_1}$ acts on the subset $\{1,2,\ldots,m_1\}$, $\mc{S}_{m_2}$ acts on the subset $\{m_1+1, m_1+2,\ldots,m_1+m_2\}$, and so on) and the corresponding representation of this subgroup is $V^{\e(m_1,h_1)}\o V^{\e(m_2,h_2)}\o \ldots\o V^{\e(m_t,h_t)}$. General facts from the theory of induced representations (see for instance \cite{Se}, chapter 7) imply the result.
\end{proof}
On the next horizontal line the types are given by a unique graph:\\

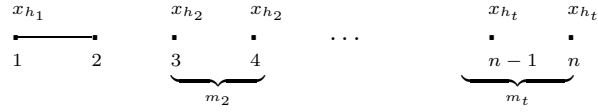
\begin{picture}(20,40)
\put(55,25){\multiput(0,0)(30,0){4}{$\centerdot$}   \put(2,1){\line(1,0){28}}    \put(120,0){$\ldots$}
\multiput(180,0)(30,0){2}{$\centerdot$}
\begin{scriptsize}
\put(0,-10){1}              \put(30,-10){2}         \put(60,-10){3}              \put(90,-10){4}
\put(180,-10){$n-1$}        \put(210,-10){$n$}      \put(0,10){$x_{h_1}$}        \put(60,10){$x_{h_2}$}
\put(90,10){$x_{h_2}$}      \put(180,10){$x_{h_t}$} \put(210,10){$x_{h_t}$}
\put(60,-12){$\mathop{\underbrace{\,\,\,\,\,\,\,\,\,\,\,\,\,\,\,\,\,\,\,\,\,\,\,\,\,\,}}\limits_{m_2}$}
\put(170,-12){$\mathop{\underbrace{\,\,\,\,\,\,\,\,\,\,\,\,\,\,\,\,\,\,\,\,\,\,\,\,\,\,\,\,\,\,\,}}\limits_{m_t}$}
\end{scriptsize}}
\end{picture}\\
(here $x_{h_2}\succ\ldots\succ x_{h_t}$, but $h_1$ could be equal to one of $h_2,\ldots,h_t$ and we have $2+m_2+\ldots+m_t=n$).

\begin{prop}\label{prop3.2}
The $\mc{S}_n$ structure corresponding to the type $L_*=(2,1^{(n-2)}), $ $H_*=(h_1; h_2^{(m_2)},\ldots,h_t^{(m_t)})$ is given by
$$E_1^{2m-1+|H_*|}(L_*,H_*)\cong\mbox{Ind}^{\mc{S}_n}_{\mc{S}_2\times \mc{S}_{m_2}\times\ldots\times
\mc{S}_{m_t}} V(2)\o V^{\e(m_2,h_2)}\o \ldots\o V^{\e(m_t,h_t)}.$$

\end{prop}
\begin{proof}
This is similar to the previous proof: the module $E_*^*(L_*,H_*)$ is the direct sum of 1-dimensional subspaces
$\mathbb{Q}\lan x_1\o \ldots\o x_nG_{ij}\ran$, where on the $i$-th position lies $x_{h_1}$, on the $j$-th position
is 1, and $x_{h_2}$ is lying on $m_2$ (arbitrary) positions, $\ldots,x_{h_t}$ on $m_t$ positions, and these subspaces
are permuted by $\mc{S}_n$. The subgroup leaving invariant the line
$\mathbb{Q}\lan x_{h_1}\o1\o\mathop{\underbrace{x_{h_2}\o x_{h_2}\o\ldots}}\limits_{m_2}\ldots\, \mathop{\underbrace{\ldots\o x_{h_t}}}\limits_{m_t}G_{12}\ran$ is the product $\mc{S}_2\times \mc{S}_{m_2}\times\ldots\times \mc{S}_{m_t}$ and the representation on this line is equivalent with
$ V(2)\o V^{\e(m_2,h_2)}\o \ldots\o V^{\e(m_t,h_t)}$, hence the result.
\end{proof}
The following types appear on the third horizontal line:\\

a) $L_*=(3,1^{(n-3)}), $ $H_*=(h_1,h_2^{(m_2)},\ldots,h_t^{(m_t)})$

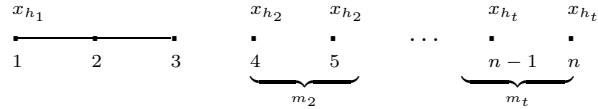
\begin{picture}(20,40)
\put(55,20){\multiput(0,0)(30,0){5}{$\centerdot$}    \put(2,1){\line(1,0){58}}    \put(150,0){$\ldots$}
\multiput(180,0)(30,0){2}{$\centerdot$}
\begin{scriptsize}
\put(0,-10){1}              \put(30,-10){2}          \put(60,-10){3}              \put(90,-10){4}
\put(120,-10){5}            \put(180,-10){$n-1$}     \put(210,-10){$n$}           \put(0,10){$x_{h_1}$}
\put(90,10){$x_{h_2}$}      \put(120,10){$x_{h_2}$}  \put(180,10){$x_{h_t}$}      \put(210,10){$x_{h_t}$}
\put(90,-12){$\mathop{\underbrace{\,\,\,\,\,\,\,\,\,\,\,\,\,\,\,\,\,\,\,\,\,\,\,\,\,\,\,\,\,\,}}\limits_{m_2}$}
\put(170,-12){$\mathop{\underbrace{\,\,\,\,\,\,\,\,\,\,\,\,\,\,\,\,\,\,\,\,\,\,\,\,\,\,\,\,\,\,\,}}\limits_{m_t}$}
\end{scriptsize}}
\end{picture}\\

b) $L_*=(2^{(2)},1^{(n-4)}), $ $H_*=(h_1^{(2)},h_2^{(m_2)},\ldots,h_t^{(m_t)})$

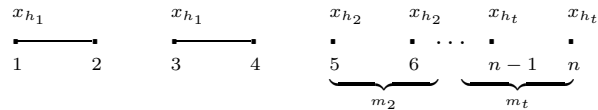
\begin{picture}(20,40)
\put(55,20){\multiput(0,0)(30,0){6}{$\centerdot$}    \multiput(2,1)(60,0){2}{\line(1,0){28}}        \put(160,0){$\ldots$}
\multiput(180,0)(30,0){2}{$\centerdot$}
\begin{scriptsize}
\put(0,-10){1}              \put(30,-10){2}          \put(60,-10){3}         \put(90,-10){4}        \put(120,-10){5}
\put(150,-10){6}            \put(180,-10){$n-1$}     \put(210,-10){$n$}      \put(0,10){$x_{h_1}$}  \put(60,10){$x_{h_1}$}
\put(120,10){$x_{h_2}$}     \put(150,10){$x_{h_2}$}  \put(180,10){$x_{h_t}$} \put(210,10){$x_{h_t}$}
\put(120,-12){$\mathop{\underbrace{\,\,\,\,\,\,\,\,\,\,\,\,\,\,\,\,\,\,\,\,\,\,\,\,\,\,\,\,\,\,}}\limits_{m_2}$}
\put(170,-12){$\mathop{\underbrace{\,\,\,\,\,\,\,\,\,\,\,\,\,\,\,\,\,\,\,\,\,\,\,\,\,\,\,\,\,\,\,}}\limits_{m_t}$}
\end{scriptsize}}
\end{picture}\\

c) $L_*=(2^{(2)},1^{(n-4)}), $ $H_*=(h_1,h_2,h_3^{(m_3)},\ldots,h_t^{(m_t)})$

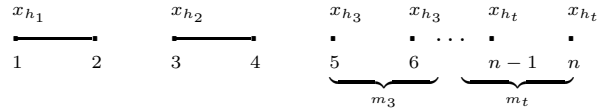
\begin{picture}(20,50)
\put(55,25){\multiput(0,0)(30,0){6}{$\centerdot$}   \multiput(2,1)(60,0){2}{\line(1,0){28}}        \put(160,0){$\ldots$}
\multiput(180,0)(30,0){2}{$\centerdot$}
\begin{scriptsize}
\put(0,-10){1}              \put(30,-10){2}         \put(60,-10){3}         \put(90,-10){4}        \put(120,-10){5}
\put(150,-10){6}            \put(180,-10){$n-1$}    \put(210,-10){$n$}      \put(0,10){$x_{h_1}$}  \put(60,10){$x_{h_2}$}
\put(120,10){$x_{h_3}$}     \put(150,10){$x_{h_3}$} \put(180,10){$x_{h_t}$} \put(210,10){$x_{h_t}$}
\put(120,-12){$\mathop{\underbrace{\,\,\,\,\,\,\,\,\,\,\,\,\,\,\,\,\,\,\,\,\,\,\,\,\,\,\,\,\,\,}}\limits_{m_3}$}
\put(170,-12){$\mathop{\underbrace{\,\,\,\,\,\,\,\,\,\,\,\,\,\,\,\,\,\,\,\,\,\,\,\,\,\,\,\,\,\,\,}}\limits_{m_t}$}
\end{scriptsize}}
\end{picture}

(In the first two cases $h_2,\ldots,h_t$ are disinct and $h_1$ could be one of them, in the third case $x_{h_1}\succ x_{h_2}$, $x_{h_3}\succ\ldots\succ x_{h_t}$ and $h_1$ or $h_2$ are not necessarily different from $h_3,\ldots,h_t$).

\begin{example}
For $n\geq 6$, the module $E_2^{2(2m-1)}(L_*=(3,1^{(n-3)}),H_*=(1,1^{(n-3)}))$ has the stable decomposition $$ V(1)_n\oplus V(2)_n\oplus V(1,1)_n\oplus V(2,1)_n.$$
\end{example}
\begin{proof}
Computing directly the corresponding character, we find for an arbitrary permutation $\sigma \in \mc{S}_n$ of type $(i_1;i_2;\ldots;i_n)$
(here $i_q$ is the number of cycles of length $q$): for any triple of 1-cycles $(i)(j)(k)$, $1\leq i<j<k\leq n$, $\sigma$ fixes the
monomials $G_{ij}G_{ik}$ and $G_{ij}G_{jk}$ and this gives $2 {i_1\choose 3}$ such monomials. The 3-cycle $(i,j,k)$ changes the sign
of $G_{ij}G_{jk}$ and there is no other combination of cycles leaving $\mathbb{Q}\lan G_{ij}G_{jk}\ran$ or $\mathbb{Q}\lan G_{ij}G_{ik}\ran$ invariant.
Therefore the character is
$$\chi_{E_2^{2(2m-1)}(\bullet - \bullet-\bullet\,\,\,\bullet\ldots\bullet )}(i_1;\ldots;i_n)=2{i_1\choose 3}-i_3.$$
Using Frobenius formula we obtain the next results:
\begin{center} \begin{tabular}{|c|c|}
   \hline
            & $\chi_{V}(i_1;\ldots;i_n)$   \\ \hline
   $V(1)_n$ & $i_1-1$                      \\ \hline
   $V(1,1)_n$ & ${i_1-1\choose 2}-i_2$    \\ \hline
   $V(2)_n$ & $\frac{1}{2}i_1(i_1-3)+i_2 $   \\\hline
   $V(3)_n$ & $\frac{1}{6}i_1(i_1-1)(i_1-5)+i_2(i_1-1)+i_3$      \\ \hline
   $V(2,1)_n$ & $\frac{1}{3}i_1(i_1-2)(i_1-4)-i_3$  \\ \hline
   $V(3,1)_n$ & $\frac{1}{8}i_1(i_1-1)(i_1-3)(i_1-6)+i_2{i_1-1\choose 2}-{i_2\choose 2}-i_4 $  \\
   \hline
\end{tabular}
\end{center}
and as a consequence the above decomposition.
\end{proof}
In the unstable cases, $n=3,4,5$, this module decomposes as
\[
\begin{array}{cl}
 n=3:  & V(2,1), \\
 n=4:  & V(3,1)\oplus V(2,2)\oplus V(2,1,1), \\
 n=5:  & V(4,1)\oplus V(3,2)\oplus V(3,1,1).
\end{array}\]
The same decompositions appear if $H_*$ is replaced by $H_*=(h_1,h_2^{(n-3)})$ with $|x_{h_1}|$, $|x_{h_2}|$ even.
\begin{example}
The module $E_2^{2(2m-1)}(L_*=(2^{(2)},1^{(n-4)}),H_*=(1^{(2)};1^{(n-4)}))$, for $n\geq 7$,  has the stable decomposition
$$V(1)_n\oplus V(2)_n\oplus V(1,1)_n\oplus V(3)_n\oplus V(2,1)_n\oplus V(3,1)_n.$$
\end{example}
\begin{proof}
Direct computation of the character gives non zero contributions only for factors of the form
$(i)(j)(k)(l)$, $(i)(j)(k,l)$, $(i,j)(k,l)$ and $(i,j,k,l)$; the result is
$$\chi_{E_*^*(\bullet-\bullet\,\,\bullet-\bullet\,\,\bullet\,\,\bullet\ldots \bullet)}(i_1;\ldots;i_n)=3{i_1\choose 4}+i_2{i_1\choose 2}-{i_2\choose2}-i_4$$
and from the same table we obtain the above decomposition.
\end{proof}
In the unstable cases we obtain the decompositions
\[
\begin{array}{cl}
 n=4:  & V(3,1), \\
 n=5:  & V(4,1)\oplus V(3,2)\oplus V(3,1,1)\oplus V(2,2,1), \\
 n=6:  & V(5,1)\oplus V(4,2)\oplus V(4,1,1)\oplus V(3,3)\oplus V(3,2,1).
\end{array}\]
\begin{example}
In the case $n=4$ we have the decompositions
$$E_2^{2(2m-1)+2|x_h|}(\mathop{\mathop{\bullet}\limits_{1}-\mathop{\bullet}\limits_2}\limits^{x_h}\,\,
\mathop{\mathop{\bullet}\limits_3-\mathop{\bullet}\limits_4)}\limits^{x_h}\cong \left\{
\begin{array}{lc}
  V(3,1) & \mbox{if } |x_h|=\mbox{even} \\
  V(4)\oplus V(2,2) & \mbox{if } |x_h|=\mbox{odd}.
\end{array}
\right.$$
\end{example}
\begin{proof}
The even case is the same as in the previous example; for the odd case direct computation of the character gives
\begin{center}
\begin{tabular}{|c|c|c|c|c|c|}
\hline
   $\sigma$ & $id$ & (12) & (123) & (1234) & (12)(34) \\ \hline
  $\chi(\sigma)$ & 3 & 1 & 0 & 1 & 3 \\  \hline
  \end{tabular}
\end{center}
For instance:
$$
\begin{array}{l}
    (1234)p_1^*(x_h)p_2^*(x_h)G_{13}G_{24}             = p_2^*(x_h)p_3^*(x_h)G_{24}G_{13} =\\
    =-p_2^*(x_h)p_3^*(x_h)G_{13}G_{24}                 =-p_2^*(x_h)p_1^*(x_h)G_{13}G_{24}=  \\
    =(-1)^{|x_h|+1}p_1^*(x_h)p_2^*(x_h)G_{13}G_{24}.
  \end{array}
$$\end{proof}
Using some of these particular cases and the same proof as in Proposition \ref{prop3.1} and \ref{prop3.2}, we obtain:
\begin{prop}\label{prop3.3}
The $\mc{S}_n$ structure of the modules corresponding to the types on the third horizontal line is given by:

a) $E_2^*(L_*=(3,1^{(n-3)}),H_*=(h_1,h_2^{(m_2)},\ldots,h_t^{(m_t)}))  \cong$ \\
$\cong Ind_{\mc{S}_3\times\mc{S}_{m_2}\times\ldots\times\mc{S}_{m_t}}^{\mc{S}_n}(V(2,1)\o V^{\e(m_2,h_2)}\o\ldots\o V^{\e(m_t,h_t)});$

b) $ E_2^*(L_*=(2^{(2)},1^{(n-4)}),H_*=(h_1^{(2)},h_2^{(m_2)},\ldots,h_t^{(m_t)})) \cong\\
\left\{
\begin{array}{ll}
Ind_{\mc{S}_4\times\mc{S}_{m_2}\times\ldots\times\mc{S}_{m_t}}^{\mc{S}_n}(V(3,1)\o V^{\e(m_2,h_2)}\o\ldots\o V^{\e(m_t,h_t)}) &  |x_{h_1}|=\mbox{even}; \\
Ind_{\mc{S}_4\times\mc{S}_{m_2}\times\ldots\times\mc{S}_{m_t}}^{\mc{S}_n}((V(4)\oplus V(2,2))\o V^{\e(m_2,h_2)}\o\ldots\o V^{\e(m_t,h_t)}) &  |x_{h_1}|=\mbox{odd};
\end{array}
\right.$

c) $E_2^*(L_*=(2^{(2)},1^{(n-4)}),H_*=(h_1,h_2,h_3^{(m_3)},\ldots,h_t^{(m_t)}))\cong\\
\begin{array}{l}
\cong Ind_{\mc{S}_2\times\mc{S}_2\times\mc{S}_{m_3}\times\ldots\times\mc{S}_{m_t}}^{\mc{S}_n}
(V(2)\o V(2)\o V^{\e(m_3,h_3)}\o\ldots\o V^{\e(m_t,h_t)})\\
\cong Ind_{\mc{S}_4\times\mc{S}_{m_3}\times\ldots\times\mc{S}_{m_t}}^{\mc{S}_n}((V(4)\oplus V(3,1)\oplus V(2,2))\o V^{\e(m_3,h_3)}\o\ldots\o V^{\e(m_t,h_t)});
\end{array}$
\end{prop}
Now we will analyze the top horizontal line; we start with the upper-left vertex of the trapezoid, $E_{n-1}^{(n-1)(2m-1)}(X,n).$
\begin{example}
For small values of $n$ the structure of this module is a consequence of the absence (for $n\geq 2$) of the one dimensional submodules of
$E_{n-1}^*(X,n)$ and of the result of Section 2: $Res_{\mc{S}_{n-1}}^{\mc{S}_n}E_{n-1}^{(n-1)(2m-1)}(X,n)\cong \mathbb{Q}[\mc{S}_{n-1}]:$
\[\begin{array}{cl}
    n=2: & E_1^{2m-1}\cong V(2); \\
    n=3: & E_2^{2(2m-1)}\cong V(2,1); \\
    n=4: & E_3^{3(2m-1)}\cong V(3,1)\oplus V(2,1,1); \\
    n=5: & E_4^{4(2m-1)}\cong V(4,1)\oplus V(3,2)\oplus V(3,1,1)\oplus V(2,2,1)\oplus V(2,1,1,1).
  \end{array}
\]
\end{example}
By Poincar\'e duality, the right side of the trapezoid has the same structure.
\begin{example}
For $n=6$ there are seven $\mc{S}_6$-modules without one dimensional submodules and satisfying
$Res_{\mc{S}_5}^{\mc{S}_6}E_5^{5(2m-1)}\cong \mathbb{Q}[\mc{S}_5].$ A direct combinatorial computation of the
character of $E_{5} ^{5(2m-1)}$ will give its non zero values
\begin{center}
\begin{tabular}{|c|c|c|c|c|}
  \hline
  $\sigma$ & $id$ & $(123456)$ & $(123)(456)$ & $(12)(34)(56)$ \\ \hline
  $\chi(\sigma)$ & $120$ & $-1$ & $-3$ & $8$ \\
  \hline
\end{tabular}
\end{center}
and therefore the next (asymmetric) decomposition: $E_5^{5(2m-1)}\cong$
\[\begin{array}{ccccccccc}
      V(5,1)      &  \oplus  & 2V(4,2)   & \oplus & V(4,1,1)    &  \oplus  &          & \oplus    &  3V(3,2,1) \\
      V(2,1,1,1,1)&  \oplus  & V(2,2,1,1)& \oplus & 2V(3,1,1,1) &  \oplus  & V(2,2,2).&
  \end{array}
\]

\noindent One can find a character table of $\mc{S}_6$ in \cite{L}.
\end{example}

Alternatively, we can use the next general result of Stanley \cite{St} and Lehrer-Solomon \cite{LeSo} and the Frobenius reciprocity formula.
Let us remind some notation necessary to present Lehrer-Solomon theorem:
consider the subgroup generated by the $n$-cycle $c_n=(1,2,\ldots ,n)$, $\lan c_n\ran$, the elementary character $\varphi_n$ of this cyclic group,  $\varphi_n(c_n)=e^{\frac{2\pi i}{n}}$ and $\varepsilon_n$, the sign character of $\mc{S}_n$.
More generally, for a partition $L_*\vdash n$, $L_*=(\l_1\geq\l_2\geq\ldots\l_t\geq1)$, let us denote by $c_{L_*}$ the product of cycles
$c_{\l_1}c_{\l_2}\ldots c_{\l_t}$, $c_{\l_1}=(1,2,\ldots,L_1),$ $c_{\l_2}=(L_1+1,L_1+2,\ldots,L_2),\ldots, c_{\l_t}=(L_{t-1}+1,L_{t-1}+2,\ldots,L_t)$.
Associated to $L_*$ we have the next diagram of groups and characters of one-dimensional representations:
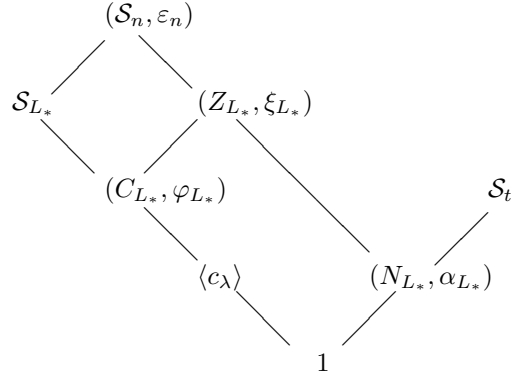
\begin{center}
\begin{picture}(10,160)
\put(10,0){$1$} \multiput(20,10)(35,30){2}{\line(1,1){20}}  \put(0,10){\line(-1,1){20}}  
\put(-35,33){$\lan c_{\l}\ran$}         \put(30,33){$(N_{L_*},\a_{L_*})$}
\put(-70,66){$(C_{L_*},\varphi_{L_*})$} \put(-35,99){$(Z_{L_*},\xi_{L_*})$} \put(75,66){$\mc{S}_t$}
\put(-105,99){$\mc{S}_{L_*}$}  \put(-70,132){$(\mc{S}_n,\ve_n)$}  \put(30,45){\line(-1,1){50}}
\multiput(-35,42)(-39,33){2}{\line(-1,1){20}}  \put(-37,107){\line(-1,1){20}}   
\put(-70,125){\line(-1,-1){20}}
\put(-35,95){\line(-1,-1){20}}
\end{picture}
\end{center}
where $\lan c_{L_*}\ran$ is the subgroup generated by $c_{L_*}$, $\mc{S}_{L_*}$ is the direct product
$$\mc{S}_{L_*}\cong\mc{S}_{\l_1}\times\mc{S}_{\l_2}\times\ldots\times\mc{S}_{\l_t},$$
$C_{L_*}$ is the centralizer of $c_{L_*}$ in $\mc{S}_{L_*}$:
$$C_{L_*}=\lan c_{\l_1}\ran\times\lan c_{\l_2}\ran\times\ldots\times\lan c_{\l_t}\ran.$$
The group $N_{L_*}$ is generated by the elements $v_i$, ``block transpositions" corresponding to equal parts $\l_i=\l_{i+1}$ in the partition $L_*$:
$$v_i=(L_{i-1}+1,L_i+1)(L_{i-1}+2,L_i+2)\ldots (L_i=L_{i-1}+\l_i,L_{i+1}=L_i+\l_{i+1}).$$
The last group, $Z_{L_*}$, is the centralizer of $c_{L_*}$ in $\mc{S}_n$ and it is a semi direct product
$$Z_{L_*}=C_{L_*}\rtimes N_{L_*}.$$
Lehrer-Solomon main formula involves the next characters:
\[\begin{array}{cl}
    \varphi_{L_*}:C_{L_*}\ra \mathbb{C}^*, & \varphi_{L_*}=(\varphi_{\l_1}\otimes\ldots\otimes \varphi_{\l_t})\cdot\varepsilon_n\mid_{C_{L_*}}, \\
    \a_{L_*}:N_{L_*}\ra\mathbb{C}^*, & \a_{L_*}(v_i)=(-1)^{\l_i+1}, \\
    \xi_{L_*}:Z_{L_*}\ra \mathbb{C}^*, & \xi_{L_*}=\a_{L_*}\cdot\varphi_{L_*}.
  \end{array}
\]
\begin{theorem}(Stanley \cite{St}, Lehrer-Solomon \cite{LeSo})
The representation of $\mc{S}_n$ on the top component $\mc{A}^{n-1}(n)$ of the Arnold algebra has the character:
$$\chi_{\mc{A}^{n-1}(n)}=\varepsilon_n Ind_{\lan c_n \ran}^{\mc{S}_n}(\varphi_n).$$
\end{theorem}
On the upper horizontal line the types are parameterized by the $n-$monotonic bamboo and the monomials from the fixed basis $\mc{B}$:
\begin{center}
\begin{picture}(50,40)
\put(-20,0){ \put(20,20){\line(1,0){80}}   \multiput(20,19)(20,0){2}{$\centerdot$}   \put(100,19){$\centerdot$}
\put(15,10){$1$}     \put(39,10){$2$}      \put(95,10){$n$}   \put(65,10){$\ldots$}  \put(19,30){$x_h$}   }
\end{picture}
\end{center}
\begin{prop}\label{prop3.4}
The $\mc{S}_n$ structure of the top component is given by
\[\begin{array}{cl}
    \chi_{E_{n-1}^{(n-1)(2m-1)+|x_h|}(n,h)} & \cong \varepsilon_n Ind_{\lan c_n \ran}^{\mc{S}_n}(\varphi_n), \\
    \chi_{E_{n-1}^{(n-1)(2m-1)+i}(X,n)} & \cong \b_i\cdot \varepsilon_n Ind_{\lan c_n \ran}^{\mc{S}_n}(\varphi_n)
  \end{array}
\]
($\b_i$ is the $i$-th Betti number).
\end{prop}
\begin{proof}
As the action of the symmetric group does not change the coefficients $$\sigma(x_h\o 1\o\ldots \o 1G_{I_*J_*})=p_{\sigma(1)}^*(x_h)\cdot\sigma(G_{I_*J_*})=p_{1}^*(x_h)\cdot\sigma(G_{I_*J_*})$$
we have the $\mc{S}_n$-decomposition: $$E_{n-1}^{(n-1)(2m-1)+i}(X,n)\cong \mathop{\bigoplus}\limits_{x_h\in\mc{B}\cap H^i(X)}p_1^*(x_h)\cdot \mc{A}^{n-1}(n).$$
\end{proof}
\begin{example}
If the number of points is an odd prime number $p$, then the multiplicities $m_{\lambda}$ of the irreducible $\mc{S}_p$-modules are given by
$$E^{(p-1)(2m-1)+|x_h|}_{p-1}(p,h)\cong \mathop{\bigoplus}\limits_{\lambda \vdash p}m_{\lambda}V(\lambda),$$
$$m_{\lambda}=\frac{1}{p}(\dim V(\lambda)-\chi_{\lambda}(c_p)).$$
\end{example}
\begin{proof}
From the Frobenius reciprocity formula we obtain the expansion (here $V(\l^{\varepsilon})=V(1,1,\ldots,1)\o V(\l)$ with character $\chi_{\l^{\ve}}$)
\[\begin{array}{cll}
    m_{\l} & =\lan\chi_V,\ve_p Ind_{\lan c_p\ran}^{\mc{S}_p}(\varphi_p)\ran_{\mc{S}_p} & =\lan\chi_{\l^\ve}, Ind_{\lan c_p\ran}^{\mc{S}_p}(\varphi_p)\ran_{\mc{S}_p}= \\
           & =\lan Res_{\lan c_p\ran}^{\mc{S}_p}\chi_{\l^{\ve}},(\varphi_p)\ran_{\mathbb{Z}_p} & =\frac{1}{p}\mathop{\sum}\limits_{k=0}^{p-1}\chi_{\l^{\ve}}(c_p^k)e^{\frac{2k\pi i}{p}}
\end{array}
\]
in which all the values $\chi_{\l^{\ve}}(c_p^k)$ are equal but not the first one:
$$\chi_{\l^{\ve}}(c_p^0)=\dim V(\l^{\ve})=\dim V(\l).$$
\end{proof}
Now we consider the type $L_*=(\l_1,\l_2,\ldots,\l_t)$, $H_*=(1^{(t)})$ ($|L_*|=\sum(\l_i-1)$). Translated into the language of types, the main formula of
Lehrer-Solomon is
\begin{theorem}(Lehrer-Solomon \cite{LeSo})
The representation of $\mc{S}_n$ on the component $E_{|L_*|}^{|L_*|(2m-1)}(L_*,1^{(t)})$ has the character
$$\chi_{E_{|L_*|}^{|L_*|(2m-1)}(L_*,1^{(t)})}=Ind_{Z_{L_*}}^{\mc{S}_n}(\xi_{L_*}).$$
\end{theorem}
To extend this to a general type $L_*=(\l_1,\ldots,\l_t)$, $H_*=(h_1,h_2,\ldots,h_t),$ we modify the notation of Lehrer-Solomon as follows:
the group $N_{(L_*,H_*)}$ is generated by the elements $v_i$ corresponding to the transposition of equal marked bamboos: $\l_i=\l_{i+1}$, $h_i=h_{i+1}$
(remember that for equal lengths $\l_c=\l_{c+1}=\ldots=\l_d$, the corresponding marks are decreasing, not necessarily strictly:
$x_{h_c}\succeq x_{h_{c+1}}\succeq\ldots\succeq x_{h_d}$); of course, $v_i$ are given by the same product of disjoint transpositions.

The subgroup $Z_{(L_*,H_*)}$ is defined by the same formula:
$$Z_{(L_*,H_*)}=C_{L_*}\rtimes N_{(L_*,H_*)},$$
but now is, in general, smaller than the centralizer of $c_{L_*}$ in $\mc{S}_n$. The character $\varphi_{L_*}:C_{L_*}\ra\mathbb{C}^*$ is given by the same formula: for a permutation $\sigma\in\mc{S}_{L_*}=\mc{S}_{l_1}\times\mc{S}_{l_2}\times\ldots\times\mc{S}_{l_t}$, the sign of $\sigma(x_{h_1}\o 1\o\ldots\o x_{h_2}\o1\ldots G_{I_*J_*})$ is given only by the permutation of the exterior factors $G_{ij}$, like in Lehrer-Solomon definition of $\varphi_{L_*}$. The coefficients $\ldots\o x_{h_i}\o1\ldots\o x_{h_i}\o1\o\ldots$ do have a contribution to the sign after the action of a permutation $\rho\in N_{(L_*,H_*)}$ if the degree $|x_{h_i}|$ is odd; therefore the character $\a_{(L_*,H_*)}$ should be modified as follows:
\[\a_{(L_*,H_*)}=\left\{
\begin{array}{lc}
  (-1)^{\l_i+1} & \mbox{if } |x_{h_i}| \mbox{ is even,} \\
  (-1)^{\l_i} &  \mbox{if } |x_{h_i}| \mbox{ is odd,}
\end{array}
\right.\]
and accordingly the character $\xi$ is modified: $$\xi_{(L_*,H_*)}=\varphi_{L_*}\cdot\a_{(L_*,H_*)}.$$ Finally we obtain the character of the Kri\v{z} algebra $E_*^*(X,n):$
\begin{theorem}\label{thm3.15}
a) The $\mc{S}_n$ representation of the submodule $E_*^*(L_*,H_*)$ has the character $$\chi_{E_*^*(L_*,H_*)}=Ind_{Z_{(L_*,H_*)}}^{\mc{S}_n}(\xi_{(L_*,H_*)}).$$

b) The $\mc{S}_n$ representations of the component $E_q^k(X,n)$ has the character $$\chi_{E_q^k(X,n)}=\mathop{\mathop{\sum}\limits_{|L_*|=q}}\limits_{|H_*|+q(2m-1)=k}Ind_{Z_{(L_*,H_*)}}^{\mc{S}_n}(\xi_{(L_*,H_*)}).$$

c) The $\mc{S}_n$ representations of the Kri\v{z} algebra $E_*^*(X,n)$ has the character $$\chi_{E_*^*(X,n)}=\mathop{\sum}\limits_{(L_*,H_*)}Ind_{Z_{(L_*,H_*)}}^{\mc{S}_n}(\xi_{(L_*,H_*)}).$$
\end{theorem}
\begin{proof}
For a given partition $ \Lambda_*=(\Lambda_1,\Lambda_2,\ldots,\Lambda_t) $ of the set $\{1,2,\ldots,n\}$,
where $ |\Lambda_i|=\l_i$, let us denote by $E_*^*(\Lambda_*,H_*)$ the span of monomials associated to
the marked graphs having $ \Lambda_*$ the set of connected components and $ H_* $ the corresponding marks. It
is obvious that the symmetric group $\mc{S}_n$ acts transitively on the components of the direct sum
$$ E_*^*(L_*,H_*)=\mathop{\oplus}\limits_{\Lambda_*}E_*^*(\Lambda_*,H_*). $$
Now we fix the term $ E_*^*(\Lambda_*,H_*) $ with
$$ \Lambda_*=(\{1,\ldots,L_1\},\{L_1+1,\ldots,L_2\},\ldots,\{L_{t-1}+1,\ldots,L_t\}) $$
and we follow the proof of Lehrer-Solomon \cite{LeSo} Section 4: the subgroup of $\mc{S}_n$ leaving
$E_*^*(\Lambda_*,H_*)$ invariant is  $ Z_{(L_*,H_*)}=C_{L_*}\rtimes N_{(L_*,H_*)}$ (from the set $v_i$ of
permutations of connected components of equal size we have to consider only the permutations of components
with the same mark). The action of this subgroup onto $E_*^*(\Lambda_*,H_*)$ is a composition of:

C) the action of $ C_{L_*}$ on each component separately, and this is given componentwise by \ref{prop3.4}; its
character is $ \varphi_{L_*}$.

N) the action of $ N_{(L_*,H_*)} $ which permutes identical marked trees; the rules of changing the sign were
already explained.

The second and the third formulae of the theorem are direct consequences of the first one.
\end{proof}

\section{Proofs of proposition \ref{prop4} and proposition \ref{prop5} \label{section4}}

In this section we will give a proof of Propositions \ref{prop4} and \ref{prop5}. From the last section in \cite{AAB} we will use the $\mc{S}_n$-decomposition of $E_1^{2m-1}(X,n)\cong \mc{A}^1(n)$ and also the bases for the irreducible $\mc{S}_n$-submodules described in this paper.
\begin{prop}\cite{AAB}
 The structure of the $\mc{S}_n$-module $E_1^{2m-1}(X,n)$ is given by
 \[
 \begin{array}{rcl}
   E_1^{2m-1}(X,2) & \cong & V(2), \\
   E_1^{2m-1}(X,3) & \cong & V(3)\oplus V(2,1), \\
   E_1^{2m-1}(X,n) & \cong & V(n)\oplus V(n-1,1)\oplus V(n-2,2)\,\,\,\,\,\mbox{for}\,\,n\geq4.
 \end{array}
 \]

 \end{prop}
We choose one nonzero element from each $\mc{S}_n$-submodule:
\[
\begin{array}{rcl}
  G^n        & = & \mathop{\sum}\limits_{i<j}G_{ij}\,\,\,\,\,\,\,\,\,\,\,\mbox{in}\,\,V(n), \\
  G^{n}_{12} & = &\mathop{\sum}\limits_{k\geq3}(G_{1k}-G_{2k})\,\,\,\,\,\,\,\,\,\,\,\mbox{in}\,\,V(n-1,1),  \\
  G_{1234}   & = & G_{14}-G_{13}+G_{23}-G_{24}  \,\,\,\,\,\,\,\,\,\,\,\mbox{in}\,\,V(n-2,2)
\end{array}
\]
and by direct computation we find non-zero differentials:

\begin{lem}
The images of the elements $G^n$ and $G^n_{12}$ under the composition $$E_1^{2m-1}\mathop{\longrightarrow}\limits^{d}E_0^{2m}\mathop{\longrightarrow}\limits^{pr}\mathop{\bigoplus}\limits_{i=1}^{n}p_i(H^{2m})$$ are given by
\[
\begin{array}{rcl}
  G^n & \mapsto & (n-1)\mathop{\sum}\limits_{i=1}^{n}p^*_i(w) \\
  G_{12}^n & \mapsto & (n-2)p^*_1(w)-(n-2)p^*_2(w).
\end{array}
\]
\end{lem}
\begin{lem}\label{lem4}
Let $x$ and $y$ be cohomology classes of positive degree ($x$ in $\mc{B}$ and $y$ in the dual bases $\mc{B}^*$) such that $xy=w$. The
image of the element $G_{1234}$ under the composition
$$E_1^{2m-1}\mathop{\longrightarrow}\limits^d E_0^{2m}\mathop{\longrightarrow}\limits^{pr}\mathop{\bigoplus}\limits_{i<j}p^*_i(\mathbb{Q}(x))p^*_j(\mathbb{Q}(y))$$
is given by $$G_{1234}\mapsto p^*_1(x)p^*_4(y)-p^*_1(x)p^*_3(y)+p^*_2(x)p^*_3(y)-p^*_2(x)p^*_4(y).$$
\end{lem}
\begin{prop}\label{prop 5.4}
If $X$ is a projective manifold ($X\neq \mathbb{C}P^1$), then the ``first" differential is injective: $$d: E_1^{2m-1}(X,n)\rightarrowtail E_0^{2m}(X,n).$$
\end{prop}
\begin{proof}
Using twice the Schur lemma for the $\mc{S}_n$-morphisms $$V(n),V(n-1,1)\hookrightarrow E_1^{2m-1}(X,n)\stackrel{d}{\to} E_0^{2m}(X,n)$$ these two submodules have isomorphic images through $d$ in $E_0^{2m}$ and trivial kernels because these morphisms are non-zero: $d(G^n)\neq 0$, $d(G_{12}^n)\neq 0.$ If $X$ is of complex dimension $m$ greater than two, we can take in Lemma \ref{lem4} $x$ to be the K\"{a}hler class and $y=x^{m-1}$; if $X$ is a projective curve, but not the projective line, the equation $xy=w$ has also non trivial solutions.
\end{proof}
\begin{rem}\label{rem4}
In the remaining case, the differential $$d:E_1^1(\mathbb{C}P^1,n)\ra E_0^2(\mathbb{C}P^1,n)$$ is injective only for $n=2,3$; for $n\geq 4$ we obtain
$$H_1^1(F(\mathbb{C}P^1,n))\cong V(2)_n.$$
We will see in the last section that in the case of $\mathbb{C}P^1$, $n\geq 4$, the differential $E_q^q(\mathbb{C}P^1)\stackrel{d}{\to} E_{q-1}^{q+1}(\mathbb{C}P^1)$ has a nontrivial kernel for $q=1,2,\ldots, n-3$ and it is injective for $q=n-2,n-1$.
\end{rem}
\begin{rem} As a consequence of \cite{LaSt}, the Proposition \ref{prop 5.4} and the Proposition \ref{prop4} are true for any formal space $X$ whose rational cohomology satisfies Poincar\'{e} duality but not for cohomology spheres (for $n\geq 4$).
\end{rem}
Now we can give a proof of Proposition \ref{prop4} by induction on $q$.
\begin{proof} ({\textit{of the Proposition \ref{prop4}}})
Let us suppose that the differential $$d: E_{q-1}^{(q-1)(2m-1)}\ra E_{q-2}^{(q-1)(2m-1)+1}$$ is injective and let $u\in E_{q}^{q(2m-1)}$ be a non zero cocycle. Let $G_{ij}$ be the smallest exterior generator (in the reverse lexicographic order $G_{12}<G_{13}<G_{23}<G_{14}<\ldots <G_{n-1,n}$) which appears in a non zero monomial (of the canonical basis) in $u$: $u=G_{ij}y+z,$ $y\in E_{q-1}(G_{\a \beta}> G_{ij})\setminus \{0\}, z\in E_q(G_{\a \beta}> G_{ij}).$ In the right hand side of the equation $$0=d(u)=-G_{ij}dy+p^*_{ij}(\Delta)y+dz$$ the last two terms, $p^*_{ij}(\Delta )y$ and $dz$, the monomials in the canonical basis contain only factors $G_{\a\b}>G_{ij}$; therefore $dy=0$ by induction on $q$, and $y=0$, and this gives a contradiction.
\end{proof}
\begin{proof}(\textit {of the Proposition \ref{prop5}})
The monomials from the canonical basis lying on the top horizontal line are $p_1^*(x_h)G_{12}G_{i_23}\ldots G_{i_{n-1}n},$ where $|x_h|=i\in[0,2m]$ and $i_a\leq a.$ The next composition is an isomorphism: $$E_{n-1}^{(n-1)(2m-1)+i}\mathop{\ra}\limits^{d}E_{n-2}^{(n-1)(2m-1)+i+1}\mathop{\ra}\limits^{pr} \mathop{\mathop{\bigoplus}\limits_{I_*=(i_2,\ldots,i_{n-1})}}\limits_{i_a\leq a}p_1^*(H^i)p_2^*(H^{2m})G_{i_23}\ldots G_{i_{n-1}n}$$
$$p_1^*(x_h)G_{12}G_{i_23}\ldots G_{i_{n-1}n}\mapsto p_1^*(x_h)p_2^*(w)G_{i_23}\ldots G_{i_{n-1}n},$$ therefore the differential is injective on the top horizontal line.
\end{proof}

\section{An acyclic subalgebra of the Kri\v{z} model}\label{5}

In \cite{BMP} is introduced a quotient of the Kri\v{z} model, denoted by $J_n$, which is quasi isomorphic to $E_*^*(X,n)$, but the
corresponding kernel is not $\mc{S}_n$-stable. We will identify an acyclic subcomplex of $E_*^*(X,n)$, denoted by $E_*^*(w(X,n))$ (or
simply by $E_*^*(w)$), which is also an $\mc{S}_n$-submodule and a subalgebra, giving another smaller complex quasi isomorphic to the Kri\v{z} model:
$$SE_*^*(X,n)=E_*^*\diagup E_*^*(w),\,\,\,H^*(E_*^*(X,n))\cong H^*(SE_*^*(X,n)).$$
The last isomorphism is now $\mc{S}_n$-equivariant.

We start with a well known result in the theory of hyperplane arrangements, see for example \cite{OT}.
\begin{defn}
The Arnold differential algebra $(\mc{A}^*(n),\partial)$ is defined by
$$\mc{A}^*(n)=\bigwedge(G_{ij},1\leq i<j\leq n)\diagup (G_{ij}G_{ik}-G_{ij}G_{jk}+G_{ik}G_{jk})$$
(the generators $G_{ij}$ have degree 1) and the differential of degree -1 is given by $\partial G_{ij}=1.$
\end{defn}
\begin{prop}
The Arnold algebra $(\mc{A}^*(n),\partial)$ is acyclic.
\end{prop}
\begin{proof}
Define the homotopy $h:\mc{A}^*\ra \mc{A}^{*+1}$ by $h(\gamma)=G_{12}\gamma$ and verify that $\partial h+h\partial=id_{\mc{A}^*}$.
\end{proof}
We denote by $E_*^{Top}(X,n)$ the submodule of the Kri\v{z} model $(E_*^*(X,n),d)$ given by the sum of the submodules of maximal total degree in
each $q$-exterior degree
$$E_*^{Top}(X,n)=\mathop{\bigoplus}\limits_{q=0}^{n-1} E_q^{Top}(X,n)=\mathop{\bigoplus}\limits_{q=0}^{n-1} E_q^{2mn-q}(X,n).$$
It is obvious that $E_*^{Top}(X,n)$, the right side of the trapezoid, is a subcomplex and an ideal of $E_*^*(X,n)$.
\begin{prop}\label{prop5.2}
There is an isomorphism of chain complexes of $\mc{S}_n$-modules $$(\mc{A}^*(n),\partial)\cong(E_*^{Top}(X,n),d).$$ In particular $$H^*(E^{Top}_*(X,n),d)=0.$$
\end{prop}
\begin{proof}
Using the standard basis $\{G_{I_*J_*}=G_{i_1j_1}G_{i_2j_2}\ldots G_{i_qj_q}\}$ in $\mc{A}^q(n)$ and the basis $\{\mathop{\prod}\limits_{h\not\in J_*}p^*_h(w)G_{I_*J_*}\}$ in $E_q^{Top}(X,n)$ (here $2\leq j_1<j_2<\ldots<j_q\leq n, 1\leq i_a<j_a$) we define the isomorphism $$f:\mc{A}^q\ra E_q^{Top}\,\,\,\mbox{by}\,\,\, f(G_{I_*J_*})=\mathop{\prod}\limits_{h\not\in J_*}p^*_h(w)G_{I_*J_*}.$$
Obviously $f$ is $\mc{S}_n$-equivariant (the degree of $w$ is even) and $f$ preserves the differentials:
\[
\begin{array}{rcl}
  df(G_{I_*J_*}) & = & d(\mathop{\prod}\limits_{h\not\in J_*}p^*_h(w)G_{I_*J_*}) \\
                 & = & \mathop{\sum}\limits_{a=1}^q (-1)^{a+1}\mathop{\prod}\limits_{h\not\in J_*}p^*_h(w)\cdot p^*_{i_aj_a}(\Delta)G_{i_1j_1}\ldots \widehat{G_{i_aj_a}} \ldots G_{i_qj_q} \\
                 & = & \mathop{\sum}\limits_{a=1}^q (-1)^{a+1}\mathop{\prod}\limits_{h\not\in J_*}p^*_h(w)\cdot p^*_{j_a}(w)G_{i_1j_1}\ldots \widehat{G_{i_aj_a}}\ldots G_{i_qj_q}  \\
                 & = & \mathop{\sum}\limits_{a=1}^q (-1)^{a+1}\mathop{\prod}\limits_{h\not\in J_*\setminus \{j_a\}}p^*_h(w)G_{i_1j_1}\ldots \widehat{G_{i_aj_a}}\ldots G_{i_qj_q} \\
                 & = & f(\mathop{\sum}\limits_{a=1}^q (-1)^{a+1}G_{i_1j_1}\ldots \widehat{G_{i_aj_a}}\ldots G_{i_qj_q})  \\
                 & = & f\partial(G_{I_*J_*}).
\end{array}
\]
For the third equality we used $i_a\not\in J_*$ and the equality $$p^*_{i_a}(w)p^*_{i_aj_a}(\Delta)= p^*_{i_a}(w)p^*_{j_a}(w).$$
\end{proof}
\begin{proof}(\textit {of the Proposition \ref{prop6}})
Now this is obvious.
\end{proof}
\noindent Now we will define three acyclic subcomplexes which generalize the previous subcomplex $(E_*^{Top}(X,n),d)$. For a fixed non empty subset $A\subset\{1,2,\ldots,n\}$ of cardinality $|A|=a\geq2$ and a fixed sequence $\b$ of length $b=n-a$, $\b=(x_1,x_2,\ldots,x_b)$, where all the elements $x_j$ belong to the fixed basis $\mc{B}$ and are different from $w$, we denote the increasing sequence of elements in $\{1,2,\ldots,n\}\setminus A$ by $b_1<b_2<\ldots<b_b$, the product $\mathop{\prod}\limits_{j=1}^{b}p_{b_j}^{*}(x_j)$ by $p^*(\b)$, and its degree $\mathop{\sum}\limits_{j=1}^{b}\deg(x_j)$ by $|\b|$. Now we define subspace $$E_{*}^{Top}(A,\b)=\mathop{\sum}\limits_{q=0}^{a-1}E_q^{2ma-q+|\b|}(A,\b)$$
by
$$E_q^{2ma-q+|\b|}(A,\b)=\mathbb{Q}\lan\mathop{\prod}\limits_{i\in A\setminus J_*}p_{i}^{*}(w)p^*(\b)G_{I_*J_*}\mid I_*\cup J_*\subset A, |J_*|=q\ran$$
(in words: the scalars in the ``complementary positions", given by $\b$, should be different from $w$, on the ''forbidden positions", corresponding to
$J_*$, there is only $1$, and all the other ``possible positions" should be filled with the top class $w$).
\begin{prop}\label{prop5.3}
For any $A$ and $\b$ as before, the space $E_*^{Top}(A,\b)$ is an acyclic subcomplex of the Kri\v{z} model.
\end{prop}
\begin{proof}
By definition $E_*^{Top}(A,\b)$ is the direct sum of its subspaces $E_q^{2ma-q+|\b|}(A,\b)$ and it is stable under the differential:
\[
\begin{array}{cc}
d(\mathop{\prod}\limits_{i\in A\setminus J_*}p_{i}^{*}(w)p^*(\b)G_{I_*J_*}) & = \mathop{\sum}\limits_{j_h\in J_*}\pm \mathop{\prod}\limits_{i\in A\setminus J_*}p_{i}^{*}(w)p_{j_h}^*(w)p^*(\b)G_{I_*\setminus\{i_h\}J_*\setminus \{j_h\}}=\\
          & = \mathop{\sum}\limits_{j_h\in J_*}\pm \mathop{\prod}\limits_{i\in A\setminus (J_*\setminus\{j_h\})}p_{i}^{*}(w)p^*(\b)G_{I_*\setminus\{i_h\}J_*\setminus \{j_h\}}.
\end{array}
\]
This subcomplex is acyclic because of the isomorphism $$(E_*^{Top}(A,\b),d)\cong (\mc{A}^*(a),\partial)$$ given by
$$\mathop{\prod}\limits_{i\in A\setminus J_*}p_{i}^{*}(w)p^*(\b)G_{I_*J_*}\leftrightarrow (-1)^{\b}G_{I_*J_*},$$
the differential of $\b$ is zero and the compatibility of the differentials $d,\partial$ was checked in the proof of \ref{prop5.3}.
\end{proof}
\begin{example}
If $A=\{1,2,\ldots,n\}$, then $\b$ is the empty sequence, $p^*(\b)=1$ and $E_*^{Top}(\{1,2,\ldots,n\},\phi)=E_*^{Top}(X,n).$
\end{example}
Now we fix a number $a$ from 2 to $n$ and a sequence $\b=(x_1,\ldots, x_b)$ of length $b=n-a$ as before. We say that the sequence $\b'=(y_1,\ldots, y_b)$ is similar to $\b$, $\b\sim \b'$, if there is a permutation $\sigma\in\mc{S}_b$ such that $y_i=x_{\sigma(i)}$ for $i=1,2,\ldots,b$. We define a new subspace: $$E_*^{Top}(a,\b)=\mathop{\mathop{\sum}\limits_{|A|=a,}}\limits_{\b'\sim\b}E_*^{Top}(A,\b').$$
\begin{prop}
For any number $a$ and a sequence $\b$ as before, the space $E_*^{Top}(a,\b)$ is an acyclic subcomplex and $\mc{S}_n$-invariant.
\end{prop}
\begin{proof}
The space is $\mc{S}_n$-invariant by construction:
$$\sigma(E_*^{Top}(A,\b))=E_*^{Top}(\sigma A,\sigma\b)=E_*^{Top}(\sigma A,\b')$$ and the acyclicity is a consequence of the direct sum decomposition:
$$(E_*^{Top}(a,\b),d)=\mathop{\mathop{\bigoplus}\limits_{|A|=a}}\limits_{\b'\sim\b}(E_*^{Top}(A,\b'),d).$$
\end{proof}
Finally, we take the whole collection of these subcomplexes: $$E_*^{*}(w(X,n))=\mathop{\sum}\limits_{a=2}^{n}\,\,\mathop{\mathop{\sum}\limits_{\b \mbox{\tiny{ of length}}}}\limits_{n-a}E_*^{Top}(a,\b).$$
\begin{prop}\label{prop5.7}
The space $E_*^*(w(X,n))$ is an acyclic, $\mc{S}_n$-invariant subcomplex.
\end{prop}
\begin{proof}
It is enough to show that the double sum is a direct sum: a monomial $x_1\o \ldots \o x_n G_{I_*J_*}$ from the canonical basis in $E_*^*(w(X,n))$ defines in a unique way the subset $A$ and the factor $\b$:
\[\begin{array}{cl}
    A & =\{i\in\{1,2,\ldots,n\}\mid x_i=w\}\cup J_*, \\
    \b & =(x_{h_1},x_{h_2},\ldots,x_{h_b}),
  \end{array}
\]
where $h_1<h_2<\ldots<h_b$ are the elements of $\{1,2,\ldots,n\}\setminus A$.
\end{proof}
\begin{prop}\label{prop5.8}
The projection map
$$ E_*^*(X,n)\longrightarrow SE_*^*(X,n)=E_*^*(X,n)/E_*^*(w(X,n))$$
is a quasi-isomorphism.
\end{prop}
\begin{proof}
This is obvious from the long exact sequence associated to
$$ 0\ra E_*^*(w(X,n))\longrightarrow E_*^*(X,n)\longrightarrow SE_*^*(X,n)\ra 0.$$
\end{proof}

\section{An example: $F(\mathbb{C}P^1,n)$}

We analyze the cohomology algebra of the configuration space of the complex projective line using the symmetric structure of the Kri\v{z} model.
We encode the symmetric structure of a bigraded $\mc{S}_n$-module $H_*^*$ into the $\mc{S}_n-$\emph{Poincar\'{e} polynomial}: $$SP_{H_*^*}(t,s)=\mathop{\sum}\limits_{\l\vdash n}(\mathop{\sum}\limits_{k,q}m_{q,\l}^kt^ks^q)V(\l),$$
where $m_{q,\l}^k$ is the multiplicity of the irreducible representation $V(\l)$ in the component $H_q^k$; the double Poincar\'{e} polynomial of $H_*^*$ is a consequence of $SP_{H_*^*}$: $P_{H_*^*}(t,s)=\mathop{\sum}\limits_{k,q}(\mathop{\sum}\limits_{\l\vdash n}m_{q,\l}^k\dim V(\l))t^ks^q.$

For $n=2$ and $n=3$ we have the next tables of the symmetric group structure of the Kri\v{z} model; using the injectivity properties of the differential, we obtain the first table and for the second table we have to use the vanishing of the cohomology on the left, top and the right side and also the acyclicity of the ``interior part" $\mathop{\oplus}\limits_{|A|=2}E_*^{Top}(A,1):$
\[
\begin{array}{rll}
    V(3)\oplus V(2,1)\cong &\lan w\o1\o1G_{12},w\o1\o1G_{13}, 1\o w\o1 G_{23}\ran & \mathop{\ra}\limits_{\cong}^{d}   \\
      \mathop{\ra}\limits_{\cong}^{d} & \lan w\o w\o1,w\o1\o w, 1\o w\o w\ran : &
  \end{array}
\]

\begin{picture}(50,60)
\put(100,0){  \put(-90,20){$E_*^*(\mathbb{C}P^1,2):$}     \put(0,1){\vector(1,0){150}}             \put(1,0){\vector(0,1){50}}
\put(155,-5){$k$}         \put(-7,45){$q$}                \multiput(0,0)(0,30){2}{\multiput(0,0)(30,0){5}{$\centerdot$}}
\multiput(-1,-1)(60,0){3}{$\bullet$}                      \multiput(29,29)(60,0){2}{$\bullet$}
\begin{scriptsize}
\multiput(-4,-8)(60,0){3}{$V(2)$}  \put(53,-14){$V(1,1)$} \multiput(34,30)(60,0){2}{$V(2)$}
\end{scriptsize}
\multiput(35,27.5)(60,0){2}{\vector(1,-1){25}}            \multiput(35,27.5)(60,0){2}{\vector(1,-1){1}}  }
\end{picture}

\begin{picture}(220,130)
\put(10,50){$E_*^*(\mathbb{C}P^1,3):$}                    \put(90,20){  \put(0,1){\vector(1,0){200}}  \put(205,-5){$k$}
\put(1,0){\vector(0,1){80}}    \put(-7,75){$q$}           \multiput(0,0)(0,30){3}{\multiput(0,0)(30,0){6}{$\centerdot$}}
\multiput(-1,-1)(60,0){4}{$\bullet$}                      \multiput(29,29)(60,0){3}{$\bullet$}
\multiput(59,59)(60,0){2}{$\bullet$}
\begin{scriptsize}
\multiput(-4,-8)(60,0){4}{$V(3)$}                         \multiput(53,-14)(60,0){2}{$V(2,1)$}     \multiput(34,35)(120,0){2}{$V(3)$}
\multiput(64,60)(60,0){2}{$V(2,1)$}                       \multiput(35,28)(120,0){2}{$V(2,1)$}     \put(94,35){$2V(3)$}
\put(95,28){$2V(2,1)$}
\end{scriptsize}
\multiput(34,29)(60,0){3}{\vector(1,-1){25}}              \multiput(59,4)(60,0){3}{\vector(1,-1){2}}
\put(34,29){\vector(1,-1){2}}                             \multiput(65,57)(60,0){2}{\vector(1,-1){25}}
\multiput(65,57)(60,0){2}{\vector(1,-1){1}} }
\end{picture}

\noindent As a consequence we obtain
\begin{lem}
The non zero components of the cohomology algebra of $F(\mathbb{C}P^1,2)$ and $F(\mathbb{C}P^1,3)$ are
$$H^0_0(F(\mathbb{C}P^1,2))\cong V(2),\,\,\,\,\,\,H_0^2(F(\mathbb{C}P^1,2))\cong V(1,1),$$
$$H^0_0(F(\mathbb{C}P^1,3))\cong V(3),\,\,\,\,\,\,H_1^3(F(\mathbb{C}P^1,3))\cong V(3).$$
In particular their symmetric Poincar\'{e} polynomials are
$$P_{F(\mathbb{C}P^1,2)}(s,t)=V(2)+t^2V(1,1),$$
$$P_{F(\mathbb{C}P^1,3)}(s,t)=(1+st^3)V(3).$$
\end{lem}
\begin{cor}
The Poincar\'{e} polynomials of the unordered configuration spaces of the projective line are
\[\begin{array}{l}
    P_{C(\mathbb{C}P^1,2)(t)}=1 \\
    P_{C(\mathbb{C}P^1,3)(t)}=1+t^3.
  \end{array}
\]
\end{cor}
Another consequence of the last computation is the fact that the Serre spectral sequences of the fibrations
$$\mc{F}_n: F(\mathbb{C},n-1)\hookrightarrow F(\mathbb{C}P^1,n)\ra \mathbb{C}P^1$$
do not degenerate at $\bold{E}_2^{*,*}$ (for $n\geq 3$): using the vanishing of the first and second cohomology of
$F(\mathbb{C}P^1,3)$ and the projection $p:F(\mathbb{C}P^1,n)\ra F(\mathbb{C}P^1,3)$, we obtain the diagram (we use
bold $\bold{E}_*^{*,*}$ for the components in the spectral sequences):

\mbox{} \mbox{} \mbox{} \mbox{} \mbox{} \mbox{} \begin{picture}(60,80)
\put(0,60){$\bold{E}_2^{0,1}(\mc{F}_3)=\mathbb{Q}\lan G_{12}\ran$}
\put(130,60){$\bold{E}_2^{0,1}(\mc{F}_n)\cong \mc{A}^1(n-1)$}
\put(90,10){$\bold{E}_2^{2,0}(\mc{F}_3)\cong \mathbb{Q}\lan w\ran$}
\put(210,10){$\bold{E}_2^{2,0}(\mc{F}_n)\cong \mathbb{Q}\lan w\ran$}
\multiput(165,13)(-80,50){2}{\vector(1,0){40}}
\multiput(60,56)(130,0){2}{\vector(2,-1){65}}
\multiput(95,40)(130,0){2}{$d_2$}
\put(75,35){$\cong$}
\multiput(185,17)(-80,50){2}{$p^*$}
\put(185,5){$\cong$}
\end{picture}

\noindent and we find that the differential $d_2$ is surjective for $n\geq3$. These spectral sequences degenerate at $E_3$: the two non zero columns are given by $$\bold{E}_{\infty}^{0,*}=\bold{E}_3^{0,*}={\ker} d_2 \cong \mc{A}^*(n-1)\diagup (G_{12}), $$
$$\bold{E}_{\infty}^{2,*}=\bold{E}_3^{2,*}\cong \bold{E}_2^{2,*}\diagup {\rm{Im}}\, d_2 \cong  \mathbb{Q}\lan w G_{12}\ran\otimes \mc{A}^*(n-1)\diagup (G_{12}). $$
\begin{prop}\cite{FZ}
The cohomology algebra of the configuration space $F(\mathbb{C}P^1,n)$ ($n\geq 3$) is given by $$H^*(F(\mathbb{C}P^1,n))\cong H^*(F(\mathbb{C}P^1,3)) \otimes \mc{A}^*(n-1)\diagup (G_{12}). $$
In particular, its Poincar\'{e} polynomial is $$P_{F(\mathbb{C}P^1,n)}(t)=(1+t^3)(1+2t)(1+3t)\ldots (1+(n-2)t).$$
\end{prop}
Using the results from section 5, we detect the nonzero bigraded components of the cohomology algebra and (partially) its $\mc{S}_n$-structure.
\begin{proof}(\textit{of the Theorem \ref{thm6}})
The first cohomology group is $$H^1(F(\mathbb{C}P^1,n))=H_1^1\cong V(n-2,2),$$ and the subalgebra generated by degree 1 elements is contained in $\mathop{\bigoplus}\limits_{q=0}^{n-3}H_q^q.$ The element
$$\gamma =2(n-2)\mathop{\sum}\limits_{i<j}p^*_i(w)G_{ij}-\mathop{\sum}\limits_{i<j}\mathop{\sum}\limits_{k\neq i,j}p^*_k(w)G_{ij}\in E_1^3(\mathbb{C}P^1,n)$$
is a cocycle in the $V(n)$-isotypic component. It can not be a coboundary because $V(n)$ is missing from $E_2^2:$
\[
\begin{array}{rcl}
  E_2^2(\mathbb{C}P^1,n) & \cong & \mc{A}^2(n)\cong \\
                         & \cong & 2V(1)_n\oplus 2V(2)_n\oplus2V(1,1)_n\oplus V(3)_n\oplus 2V(2,1)_n\oplus V(3,1)_n
\end{array}
\]
(this is correct in the stable case $n\geq 7$, see \cite{CF} or \cite{AAB};  the trivial module $V(n)$ does not appear in the unstable cases either).

As $\b_3=1+\mathop{\sum}\limits_{2\leq i<j<k\leq n-2}ijk$ and the component $H_3^3$ contains a submodule of dimension $\b_3-1$, we obtain
$$H_1^3(F(\mathbb{C}P^1,n))\cong V(n).$$
The ideal generated by $\gamma$ is contained in $\mathop{\bigoplus}\limits_{q=1}^{n-2}H_q^{q+2}$ and algebra structure shows that all the
other bigraded components are zero.
\end{proof}

The module $H_2^2$ is a quotient of $E_2^2\cong \mc{A}^2(n)$ (its decomposition into irreducible modules was given in the last proof)
and also a quotient of
$$\begin{array}{ccl}
   H_1^1\wedge H_1^1 &\cong & \bigwedge^2V(2)_n \\
                     &\cong & V(1,1)_n\oplus V(2,1)_n\oplus V(1,1,1)_n\oplus V(3,1)_n
\end{array}$$
(see \cite{AAB}); the intersection of these decompositions gives (for $n\geq 7$) the inclusion
$$H_2^2<V(1,1)_n\oplus V(2,1)_n\oplus V(3,1)_n $$
and computing their dimensions this inclusion becomes an equality:
$$\begin{array}{cll}
    \b_2 & = & \mathop{\sum}\limits_{2\leq i<j \leq n-2}ij  =  \dfrac{(n-4)(n-3)(3n^2-n+2)}{24} \\
     & = & \dfrac{(n-1)(n-2)}{2}+\dfrac{n(n-1)(n-4)}{3}+\dfrac{4(n-1)(n-3)(n-6)}{8}     \\
     & = & {\rm dim}V(1,1)_n+{\rm dim}V(2,1)_n+{\rm dim}V(3,1)_n.
  \end{array}
$$
Similar computations give the unstable cases of the next proposition.
\begin{prop}
The decomposition of the second cohomology group becomes stable for $n\geq 7$ and it is given by
$$H^2(F(\mathbb{C}P^1,n))=H_2^2\cong V(1,1)_n\oplus V(2,1)_n\oplus V(3,1)_n.$$
In the unstable cases we have
$$\begin{array}{cclr}
                                             H^2(F(\mathbb{C}P^1,n))  & =     & 0        & {\rm for} \,\, n=2,3,4; \\
                                             H^2(F(\mathbb{C}P^1,5))  & \cong & V(3,1,1) &  \\
                                             H^2(F(\mathbb{C}P^1,6))  & \cong & V(4,1,1)\oplus V(3,2,1). &
                                           \end{array}
$$

\end{prop}

\end{document}